\documentclass{amsart}

\usepackage{enumerate}
\usepackage{amsmath,amsthm,amscd,amssymb}

\usepackage{latexsym}
\usepackage{upref}
\usepackage{verbatim}

\usepackage[mathscr]{eucal}
\usepackage{dsfont}

\usepackage{graphicx}
 \numberwithin{equation}{section}

\newtheorem{theorem}{Theorem}[section]

\newtheorem{corollary}[theorem]{Corollary}
\newtheorem{lemma}[theorem]{Lemma}
\theoremstyle{definition}
\newtheorem{definition}[theorem]{Definition}
\theoremstyle{remark}
\newtheorem{remark}[theorem]{Remark}

\newtheorem{hypothesis}[theorem]{Hypothesis}

\chardef\bslash=`\\ % p. 424, TeXbook

\newcommand{\wh}{\widehat}

\newcommand{\whA}{T}
\newcommand{\whB}{T_{\cB}}
\newcommand{\whBo}{T_{\cB_0}}

\newcommand{\dA}{{\dot A}}

\newcommand{\bbR}{{\mathbb{R}}}

\newcommand{\bbC}{{\mathbb{C}}}

\newcommand{\ran}{\text{\rm{Ran}}}

\newcommand{\dom}{\text{\rm{Dom}}}

\newcommand{\calH}{{\mathcal H}}

\newcommand{\calK}{{\mathcal K}}

\newcommand{\calM}{{\mathcal M}}

\newcommand{\calR}{{\mathcal R}}

\newcommand{\linspan}{\mathrm{lin\ span}}

\newcommand{\mM}{\mathfrak M}

\renewcommand{\Im}{\text{\rm Im}}

\def\sD{{\mathfrak D}}

\def\sM{{\mathfrak M}}   \def\sN{{\mathfrak N}}

\def\bA{{\mathbb A}}   \def\dB{{\mathbb B}}   \def\dC{{\mathbb C}}

      \def\dR{{\mathbb R}}

   \def\cB{{\mathcal B}}   
      
   \def\cH{{\mathcal H}}   
      \def\cL{{\mathcal L}}

\def\RE{{\rm Re\,}}
\def\Ker{{\rm Ker\,}}

\def\wh{\hat}

\def\uphar{{\upharpoonright\,}}

\DeclareMathOperator{\IM}{Im}

\newcommand{\eval}[2][\right]{\relax
  \ifx#1\right\relax \left.\fi#2#1\rvert}

\begin{document}

\title[A system coupling and the Donoghue classes]{A system coupling and Donoghue classes of Herglotz-Nevanlinna functions}

\author[Belyi]{S. Belyi}
\address{Department of Mathematics\\ Troy State University\\
Troy, AL 36082, USA
%URL: {\sf http://spectrum.troy.edu/$\sim$belyi/}
}
\curraddr{}
\email{sbelyi@troy.edu}
%\thanks{Thank you.}

\author[Makarov]{K. A. Makarov}
\address{Department of Mathematics\\
 University of Missouri\\
  Columbia, MO 63211, USA}
\email{makarovk@missouri.edu}

\author[Tsekanovskii]{E. Tsekanovski\u i}
\address{Department of Mathematics,\\ Niagara University, NY
14109, USA}
\email{tsekanov@niagara.edu}
%\thanks{Thank you very much.}


%\dedicatory{ }


\begin{abstract}  We study  the impedance functions of conservative L-systems with the unbounded main operators. In addition to the  generalized Donoghue class $\sM_\kappa$ of Herglotz-Nevanlinna functions considered   by the authors earlier, we introduce  ``inverse" generalized Donoghue classes $\sM_\kappa^{-1}$ of functions satisfying  a different normalization condition on the generating measure, with  a criterion for the impedance function $V_\Theta(z)$ of an L-system $\Theta$ to belong to the class $\sM_\kappa^{-1}$  presented.
 In addition, we establish a connection between ``geometrical" properties of two L-systems whose impedance functions belong to the classes $\sM_\kappa$ and $\sM_\kappa^{-1}$, respectively.
In the second part of the paper we introduce  a coupling of two L-system  and show that if the impedance functions of two L-systems belong to the generalized Donoghue classes $\sM_{\kappa_1}$($\sM_{\kappa_1}^{-1}$) and $\sM_{\kappa_2}$($\sM_{\kappa_2}^{-1}$), then the impedance function of the  coupling falls into the class $\sM_{\kappa_1\kappa_2}$.
Consequently, we obtain that if an L-system whose impedance function belongs to the  standard  Donoghue class $\sM=\sM_0$ is coupled with any other L-system, the impedance function of the coupling belongs to $\sM$ (the absorbtion property).
Observing the result of coupling of $n$ L-systems as $n$ goes to infinity,
 we put forward the concept of a limit coupling which leads to the notion
of the system attractor, two models of which (in the position and momentum representations) are presented.
All major results are illustrated by  various examples.
\end{abstract}

\maketitle

\section{Introduction}\label{s1}

This article is a part of an ongoing project studying the connections between various subclasses of Herglotz-Nevanlinna functions and conservative realizations of L-systems (see \cite{AlTs2}, \cite{ABT}, \cite{BMkT}).

Let $T$ be a densely defined closed operator in a Hilbert space $\cH$ such that its resolvent set $\rho(T)$ is not empty. We also assume that
$\dom(T)\cap \dom(T^*)$ is dense and that the restriction $\dA=T|_{\dom(T)\cap \dom(T^*)}$ is a closed symmetric operator with finite and equal deficiency indices.
Let $\calH_+\subset\calH\subset\calH_-$ be the rigged Hilbert space associated with $\dot A$ (see Section \ref{s2}).

One of the main objectives  of the current paper is the study of the conservative \textit{L-system}
\begin{equation}
\label{col0}
 \Theta =
\left(%
\begin{array}{ccc}
  \bA    & K & J \\
   \calH_+\subset\calH\subset\calH_- &  & E \\
\end{array}%
\right),
\end{equation}
where the \textit{state-space operator} $\bA$ is a bounded linear operator from
$\calH_+$ into $\calH_-$ such that  $\dA \subset T\subset \bA$, $\dA \subset T^* \subset \bA^*$, $E$ is a finite-dimensional Hilbert space,
$K$ is a bounded linear operator from the space $E$ into $\calH_-$,   {and} $J=J^*=J^{-1}$ is a self-adjoint isometry  on $E$ such that the imaginary part of $\bA$ has a representation $\IM\bA=KJK^*$. Due to the facts that $\calH_\pm$ is dual to $\calH_\mp$ and  {that} $\bA^*$ is a bounded linear operator from $\calH_+$ into $\calH_-$, $\IM\bA=(\bA-\bA^*)/2i$ is a well defined bounded operator from $\calH_+$  into $\calH_-$.
Note that the main operator $T$ associated with the system $\Theta$ is uniquely determined by the state-space operator $\bA$ as its restriction on the domain $\dom(T)=\{f\in\calH_+\mid \bA f\in\calH\}$.

Recall that the operator-valued function  given by
\begin{equation*}\label{W1}
 W_\Theta(z)=I-2iK^*(\bA-zI)^{-1}KJ,\quad z\in \rho(T),
\end{equation*}
 is called the \textit{transfer function}  of the L-system $\Theta$ and
\begin{equation*}\label{real2}
 V_\Theta(z)=i[W_\Theta(z)+I]^{-1}[W_\Theta(z)-I] =K^*(\RE\bA-zI)^{-1}K,\quad z\in\rho(T)\cap(\dC\setminus \dR),
\end{equation*}
is called the \textit{impedance function } of $\Theta$.

The main  goal of the paper is to study L-systems  the impedance functions of which belong to the generalized Donoghue class
$\sM_\kappa^{-1}$
consisting  of all analytic mappings $M$ from $\bbC_+$ into itself  that admits the representation
$$
M(z)=\int_\bbR \left
(\frac{1}{\lambda-z}-\frac{\lambda}{1+\lambda^2}\right )
d\mu(\lambda),
$$
where $\mu$ is an  infinite Borel measure  such that
$$
\int_\bbR\frac{d\mu(\lambda)}{1+\lambda^2}=\frac{1+\kappa}{1-\kappa} \ge 1\quad  (0\le\kappa<1).
$$
Note that  L-systems with the impedance functions from the   generalized Donoghue class $\sM_\kappa$ have been studied earlier in   \cite{BMkT} (see eq.  \eqref{e-38-kap} below for the definition of the classes $\sM_\kappa$). A new twist in our exposition is introducing the concept  of a coupling of two L-systems associated with various generalized Donoghue classes followed by the study of analytic properties of the impedance functions of the coupling.

The paper is organized as follows.

In Section \ref{s2}  we recall the definition of an $L$-system and provide necessary background.

Section \ref{s4} is of auxiliary nature and it contains some basic facts about the Liv\v{s}ic function
associated with a pair $(\dot A, A)$ of a symmetric operator with deficiency indices $(1,1)$ and its self-adjoint extension $A$  as presented in \cite{Lv1}, \cite{Lv2}, \cite{MT-S}.

In Section \ref{s3} we present an explicit construction of two $L$-systems whose impedance functions belong to generalized Donoghue classes considered in Section \ref{s5}.

In Section \ref{s5}, following our development in \cite{BMkT}, we introduce yet another generalized Donoghue class of functions $\sM_\kappa^{-1}$ and establish a connection between ``geometrical" properties of two $L$-systems whose impedance functions belong to the classes $\sM_\kappa$ and $\sM_\kappa^{-1}$, respectively. Then (see Theorem \ref{t-22}) we present a criterion for the impedance function $V_\Theta(z)$ of an L-system $\Theta$ to be a member of the class $\sM_\kappa^{-1}$.

In Section \ref{s6} we introduce a concept  of a  coupling of two L-system (see Definition \ref{d6-7-2}) and  show (Theorem \ref{t6-7-3}) that the procedure of coupling serves as a serial connection  (the transfer function of the coupling is a simple product of the transfer functions of L-systems coupled). Moreover, it is proved  (Theorem \ref{t28}) that if the impedance functions of two L-systems belong to   the generalized Donoghue classes $\sM_{\kappa_1}$ and $\sM_{\kappa_2}$, then the impedance function of the constructed coupling falls into the class $\sM_{\kappa_1\kappa_2}$. An immediate consequence of this is the absorbtion property of the regular Donoghue class $\sM=\sM_0$: if an $L$-system whose impedance function belongs to the class $\sM$ is coupled with any other $ L$-system, the impedance function of the coupling belongs to $\sM$ (see Corollary \ref{c-29}). We also show that if the impedance functions of two L-systems belong to   the generalized Donoghue classes $\sM_{\kappa_1}^{-1}$ and $\sM_{\kappa_2}^{-1}$, then the impedance function of the coupling still falls into the class $\sM_{\kappa_1\kappa_2}$ (see Corollary \ref{c-19}). A similar result takes place if the impedance function of one L-system belongs to   the generalized Donoghue class $\sM_{\kappa_1}$ while the impedance function of the second L-system is from the class $\sM_{\kappa_2}^{-1}$. Then the coupling belongs to the class $\sM_{\kappa_1\kappa_2}^{-1}$ (see Corollary \ref{c-23}). The classes $\sM$, $\sM_k$ and $\sM_{k}^{-1}$ can be considered symbolically as classes of \textit{mass preserving}, \textit{mass decreasing} and \textit{mass increasing} classes respectively (see \eqref{hernev}, \eqref{e-38-kap}, \eqref{e-39-kap}, \eqref{e-73-remark}).

In Section \ref{s7} we put forward the concept of a limit coupling (see Definition \ref{d17}) and define  the system attractor,  two models of which (in the position and momentum representations) are discussed in  the end of the section.

We conclude the paper by providing several examples that illustrate all the main results and concepts.

For the sake of completeness, a functional model for a prime dissipative triple used in our considerations is presented in Appendix A.

\section{Preliminaries}\label{s2}

For a pair of Hilbert spaces $\calH_1$, $\calH_2$ we denote by $[\calH_1,\calH_2]$ the set of all bounded linear operators from $\calH_1$ to $\calH_2$. Let $\dA$ be a closed, densely defined, symmetric operator with finite equal deficiency indices acting on a Hilbert space $\calH$ with inner product $(f,g),f,g\in\calH$. Any operator $T$ in $\cH$ such that
\[
\dA\subset T\subset\dA^*
\]
is called a \textit{quasi-self-adjoint extension} of $\dA$.

 Consider the rigged Hilbert space (see \cite{Ber}, \cite{BT3})
%\label{107}
$\calH_+\subset\calH\subset\calH_- ,$ where $\calH_+ =\dom(\dA^*)$ and
%\label{108}
\begin{equation}\label{108}
(f,g)_+ =(f,g)+(\dA^* f, \dA^*g),\;\;f,g \in \dom(A^*).
\end{equation}
Let $\calR$ be the \textit{\textrm{Riesz-Berezansky   operator}} $\calR$ (see \cite{Ber}, \cite{BT3}) which maps $\mathcal H_-$ onto $\mathcal H_+$ such
 that   $(f,g)=(f,\calR g)_+$ ($\forall f\in\calH_+$, $g\in\calH_-$) and
 $\|\calR g\|_+=\| g\|_-$.
 Note that
identifying the space conjugate to $\calH_\pm$ with $\calH_\mp$, we
get that if $\bA\in[\calH_+,\calH_-]$, then
$\bA^*\in[\calH_+,\calH_-].$
An operator $\bA\in[\calH_+,\calH_-]$ is called a \textit{self-adjoint
bi-extension} of a symmetric operator $\dA$ if $\bA=\bA^*$ and $\bA
\supset \dA$.
Let $\bA$ be a self-adjoint bi-extension of $\dA$ and let the operator $\hat A$ in $\cH$ be defined as follows:
\[
\dom(\hat A)=\{f\in\cH_+:\hat A f\in\cH\}, \quad \hat A=\bA\uphar\dom(\hat A).
\]
The operator $\hat A$ is called a \textit{quasi-kernel} of a self-adjoint bi-extension $\bA$ (see \cite{TSh1}, \cite[Section 2.1]{ABT}).
According to the von Neumann Theorem (see \cite[Theorem 1.3.1]{ABT}) if $\wh A$ is a self-adjoint extension of $\dA$, then
\begin{equation}\label{DOMHAT}
\dom(\hat A)=\dom(\dA)\oplus(I-U)\sN_{i},
\end{equation}
where $U$ is a $(\cdot)$ (and $(+)$)-isometric operator from $\sN_i$ into $\sN_{-i}$
 and $$\sN_{\pm i}=\Ker (\dA^*\mp i I)$$ are the deficiency subspaces of $\dA$.
 A self-adjoint bi-extension $\bA$ of a symmetric operator $\dA$ is called \textit{t-self-adjoint} (see \cite[Definition 3.3.5]{ABT}) if its quasi-kernel $\hat A$ is
self-adjoint operator in $\calH$.
An operator $\bA\in[\calH_+,\calH_-]$  is called a \textit{quasi-self-adjoint bi-extension} of an operator $T$ if
$\bA\supset T\supset \dA$ and $\bA^*\supset T^*\supset\dA.$ In what follows we will be mostly interested in the following type of quasi-self-adjoint bi-extensions.
\begin{definition}[\cite{ABT}]\label{star_ext}
Let $T$ be a quasi-self-adjoint extension of $\dA$ with non\-empty
resolvent set $\rho(T)$. A quasi-self-adjoint bi-extension $\bA$ of an operator $T$ is called a \textit{($*$)-extension }of $T$ if $\RE\bA$ is a
t-self-adjoint bi-extension of $\dA$.
\end{definition}
Under the  assumption that $\dA$ has equal finite deficiency indices we say that a quasi-self-adjoint extension $T$ of $\dA$ belongs to the
\textit{class $\Lambda(\dA)$} if $\rho(T)\ne\emptyset$, $\dom(\dA)=\dom(T)\cap\dom(T^*)$, and hence  $T$ admits $(*)$-extensions. The description of
all $(*)$-extensions via Riesz-Berezansky   operator $\calR$ can be found in \cite[Section 4.3]{ABT}.

\begin{definition}
A system of equations
\[
\left\{   \begin{array}{l}
          (\bA-zI)x=KJ\varphi_-  \\
          \varphi_+=\varphi_- -2iK^* x
          \end{array}
\right.,
\]
 or an
array
%\label{141}
\begin{equation}\label{e6-3-2}
\Theta= \begin{pmatrix} \bA&K&\ J\cr \calH_+ \subset \calH \subset
\calH_-& &E\cr \end{pmatrix}
\end{equation}
 is called an \textbf{{L-system}}   if:
\begin{enumerate}
\item[(1)] {$\mathbb  A$ is a   ($\ast $)-extension of an
operator $T$ of the class $\Lambda(\dA)$};
\item[(2)] {$J=J^\ast =J^{-1}\in [E,E],\quad \dim E < \infty $};
\item[(3)] $\IM\bA= KJK^*$, where $K\in [E,\calH_-]$, $K^*\in [\calH_+,E]$, and
$\ran(K)=\ran (\IM\bA).$
\end{enumerate}
\end{definition}
In the definition above   $\varphi_- \in E$ stands for an input vector, $\varphi_+ \in E$ is an output vector, and $x$ is a state space vector in
$\calH$.
The operator $\bA$  is called the \textit{state-space operator} of the system $\Theta$, $T$ is the \textit{main operator},  $J$ is the \textit{direction operator}, and $K$ is the  \textit{channel operator}. A system $\Theta$ in \eqref{e6-3-2} is called \textit{minimal} if the operator $\dA$ is a prime operator in $\calH$, i.e., there exists no non-trivial reducing
invariant subspace of $\calH$ on which it induces a self-adjoint operator.
An L-system $\Theta$ defined above is \textit{conservative} in the sense explained in \cite[Section 6.3]{ABT}. Such systems (open systems) with bounded operators were introduced by Liv\v{s}ic \cite{Lv2} and studied in \cite{Bro}, \cite{BrLv}, \cite{Lv1}, \cite{LiYa}.

We  associate with an L-system $\Theta$ the operator-valued function
\begin{equation}\label{e6-3-3}
W_\Theta (z)=I-2iK^\ast (\mathbb  A-zI)^{-1}KJ,\quad z\in \rho (T),
\end{equation}
which is called the \textbf{transfer  function} of the L-system $\Theta$. We also consider the operator-valued function
\begin{equation}\label{e6-3-5}
V_\Theta (z) = K^\ast (\RE\bA - zI)^{-1} K, \quad z\in\rho(\hat A).
\end{equation}
It was shown in \cite{BT3}, \cite[Section 6.3]{ABT} that both \eqref{e6-3-3} and \eqref{e6-3-5} are well defined. The transfer operator-function $W_\Theta (z)$ of the system
$ \Theta $ and an operator-function $V_\Theta (z)$ of the form (\ref{e6-3-5}) are connected by the following relations valid for $\IM z\ne0$, $z\in\rho(T)$,
\begin{equation}\label{e6-3-6}
\begin{aligned}
V_\Theta (z) &= i [W_\Theta (z) + I]^{-1} [W_\Theta (z) - I] J,\\
W_\Theta(z)&=(I+iV_\Theta(z)J)^{-1}(I-iV_\Theta(z)J).
\end{aligned}
\end{equation}
The function $V_\Theta(z)$ defined by \eqref{e6-3-5} is called the
\textbf{impedance function} of an L-system $ \Theta $ of the form
(\ref{e6-3-2}). The class of all Herglotz-Nevanlinna functions in a finite-dimensional Hilbert
space $E$, that can be realized as impedance functions of an L-system, was described in \cite{BT3}, \cite[Definition 6.4.1]{ABT}.

Two minimal L-systems
$$
\Theta_j= \begin{pmatrix} \bA_j&K_j&J\cr \calH_{+j} \subset\calH_j \subset \calH_{-j}& &E\cr
\end{pmatrix},\quad j=1,2,
$$
are called \textbf{bi-unitarily equivalent} \cite[Section 6.6]{ABT} if there exists a triplet of operators $(U_+, U, U_-)$ that isometrically maps the triplet $\calH_{+1}\subset\calH_1\subset\calH_{-1}$ onto the triplet $\calH_{+2}\subset\calH_2\subset\calH_{-2}$
such that $U_+=U|_{\calH_{+1}}$ is an isometry from $\calH_{+1}$ onto $\calH_{+2}$, $U_-=(U_+^*)^{-1}$ is an isometry from $\calH_{-1}$ onto $\calH_{-2}$, and
\begin{equation}\label{167}
UT_1=T_2U,  \quad U_-\bA_1=\bA_2 U_+,\quad U_-K_1=K_2.
\end{equation}
It is shown in \cite[Theorem 6.6.10]{ABT} that if the transfer functions $W_{\Theta_1}(z)$ and $W_{\Theta_2}(z)$ of two minimal L-systems $\Theta_1$ and $\Theta_2$ are the same on
$ z\in\rho(T_1)\cap\rho(T_2)\cap \dC_{\pm}\ne\emptyset$, then $\Theta_1$ and $\Theta_2$ are bi-unitarily equivalent.

\section{The Liv\v{s}ic function}\label{s4}

Another important object of interest and a tool that we use in this paper is the characteristic function of a symmetric
operator introduced by Liv\v{s}ic in \cite{L}.
 In \cite{MT-S}  two of the authors  (K.A.M. and E.T.) suggested to define a characteristic function  of a symmetric operator as well of its dissipative extension  as the one associated with  the pairs $(\dot A, A)$  and  $(T, A)$, rather than with the single operators $\dot A$ and $T$, respectively.
Following \cite{MT-S}, \cite{MT10} we   call the characteristic function associated with the pair $(\dot A, A)$ the \textit{Liv\v{s}ic function}.
 For a detailed treatment  of the   Liv\v{s}ic function and  characteristic function   we refer to \cite{MT-S} and references therein.
Suppose that  $A$ is a self-adjoint extension of a symmetric operator
$\dot A$ with deficiency indices $(1,1)$. Let
$g_\pm$  be deficiency elements $g_\pm\in \Ker (\dot A^*\mp iI)$,
$\|g_\pm\|=1$. Assume, in addition, that
\begin{equation}\label{star}
g_+-g_-\in \dom( A).
\end{equation}
Following \cite{MT-S},   the \textit{Liv\v{s}ic function} $s(\dot A, A)$  \emph{associated with the pair} $(\dot A, A)$ is
\begin{equation}\label{charsum}
s(z)=\frac{z-i}{z+i}\cdot \frac{(g_z, g_-)}{(g_z, g_+)}, \quad
z\in \bbC_+,
\end{equation}
 Here $g_\pm\in \Ker( \dA^*\mp iI)$ are normalized appropriately chosen deficiency elements and
$ g_z\ne 0$ are arbitrary  deficiency elements of the symmetric operators $\dot A$.
The Liv\v{s}ic function $s(z)$ is
a complete unitary invariant  of a pair $(\dot A, A)$ with a prime\footnote{A symmetric operator $\dot A$ is prime
if there does not exist a subspace invariant under $\dot A$ such that the restriction of $\dot A$ to this subspace is self-adjoint.},
 densely defined symmetric operator with deficiency indices $(1,1)$ symmetric operator $\dot A$ and its self-adjoint extension $A$.

Suppose also that $T \ne (T )^*$ is a  maximal dissipative extension of $\dot A$,
$$
\Im(T f,f)\ge 0, \quad f\in \dom(T ).
$$
Since $\dot A$ is symmetric, its dissipative extension $T$ is automatically quasi-self-adjoint \cite{ABT}, %\cite{St68},
that  is,
$$
\dot A \subset T \subset \dA^*,
$$
and hence, according to \eqref{e4-60} with $\calK=\kappa$,
\begin{equation}\label{parpar}
g_+-\kappa g_-\in \dom
 (T )\quad \text{for some }
|\kappa|<1.
\end{equation}
Based on the  parametrization \eqref{parpar} of the domain of $T$ and following \cite{L}, \cite{MT-S}, the M\"obius transformation
\begin{equation}\label{ch12}
S(z)=\frac{s(z)-\kappa} {\overline{ \kappa }\,s(z)-1}, \quad z\in \bbC_+,
\end{equation}
where $s$ is given  by \eqref{charsum}, is called the \textbf{characteristic function} of the dissipative extension $T$ \cite{L}.
The characteristic function $S(z)$ of a dissipative (maximal) extension $T$ of a densely defined prime symmetric operator $\dot A$ is a complete unitary invariant of the triple $(\dot A, A, T )$ (see  \cite{MT-S}).

Recall that  Donoghue \cite{D}  introduced a concept of the Weyl-Titchmarsh function $M(\dot A, A)$ associated with a pair $(\dot A, A)$ by
$$M(\dot A, A)(z)=
((Az+I)(A-zI)^{-1}g_+,g_+), \quad z\in \bbC_+,
$$
$$g_+\in \Ker( \dA^*-iI),\quad \|g_+\|=1,
$$where $\dot A $ is
 a symmetric operator with deficiency indices $(1,1)$,
%$\text{def}(\dot A)=(1,1)$,
 and $A$ is
its self-adjoint extension.

Denote by $\mM$ the \textbf{Donoghue class} of all analytic mappings $M$ from $\bbC_+$ into itself  that admits the representation
 \begin{equation}\label{hernev}
M(z)=\int_\bbR \left
(\frac{1}{\lambda-z}-\frac{\lambda}{1+\lambda^2}\right )
d\mu,
\end{equation}
where $\mu$ is an  infinite Borel measure   and
$$
\int_\bbR\frac{d\mu(\lambda)}{1+\lambda^2}=1\,,\quad\text{equivalently,}\quad M(i)=i.
$$
It is known  \cite{D}, \cite{GMT}, \cite{GT}, \cite {MT-S} that $M\in \mM$ if and only if $M$
can be realized  as the Weyl-Titchmarsh function $M(\dot A, A)$ associated with a pair $(\dot A, A)$.

A standard relationship between the  Weyl-Titch\-marsh and the    Liv\v{s}ic functions associated with the  pair
$(\dA, A) $ was described in \cite {MT-S}. In particular, if we denote by $M=M(\dot A, A)$ and by $s=s(\dot A, A)$
  the Weyl-Titchmarsh function and the    Liv\v{s}ic function  associated with the pair $(\dot A, A)$, respectively,
then
\begin{equation}\label{blog}
s(z)=\frac{M(z)-i}{M(z)+i},\quad z\in \bbC_+.
\end{equation}

\begin{hypothesis}\label{setup} Suppose
that $\whA \ne\whA^*$  is  a maximal
dissipative extension of  a symmetric operator $\dot A$
 with deficiency indices $(1,1)$. Assume, in addition, that
$A$ is a  self-adjoint extension of $\dot A$. Suppose,  that the
deficiency elements $g_\pm\in \Ker (\dA^*\mp iI)$ are
normalized, $\|g_\pm\|=1$, and chosen in such a way that
\begin{equation}\label{ddoomm14}
g_+- g_-\in \dom ( A)\,\,\,\text{and}\,\,\,
g_+-\kappa g_-\in \dom (\whA )\,\,\,\text{for some }
\,\,\,|\kappa|<1.
\end{equation}
\end{hypothesis}

Under Hypothesis \ref{setup}, consider the characteristic
function $S=S( \dot A, \whA , A)$
associated with the triple of operators $( \dot A, \whA , A)$

\begin{equation}\label{ch1}
S(z)=\frac{s(z)-\kappa} {\overline{ \kappa }\,s(z)-1}, \quad z\in \bbC_+,
\end{equation}
 where $s=s(\dot A, A)$ is the Liv\v{s}ic function associated with the pair
$(\dot A, A)$.

We remark that given  a triple  $( \dot A, \whA , A)$, one can always find a
basis $g_\pm$ in the deficiency subspace
 $\Ker (\dA^*-iI)\dot +\Ker (\dA^*+iI)$,
$$\|g_\pm\|=1, \quad g_\pm\in (\dA^*\mp iI),
$$ such that
$$
g_+-g_-\in \dom (A)
\quad \text{
and}\quad
g_+-\kappa g_-\in \dom (\whA ),
$$
and then, in this case,
\begin{equation}\label{assa}
\kappa =S( \dot A, \whA , A)(i).
\end{equation}

It was pointed out in \cite{MT-S}, that  if $\kappa=0$, the quasi-self-adjoint extension $\whA $
coincides with the restriction of the adjoint operator $(\dot A)^*$ on
$$
\dom(\whA )=\dom(\dot A)\dot + \Ker (\dA^*-iI).
$$
and that the  prime triples $(\dot A, \whA , A)$ with
$\kappa=0$
 are  in one-to-one correspondence with the set of prime symmetric operators.
In this case, the characteristic function $S$
and the Liv\v{s}ic function $s$ coincide (up to a sign),
$$
S(z)=-s(z), \quad z\in \bbC_+.
$$

Similar constructions can be obtained if Hypothesis \ref{setup} is replaced with an ``anti-Hypothesis" as follows.
\begin{hypothesis}\label{setup-1} Suppose
that $\whA \ne\whA^*$  is  a maximal
dissipative extension of  a symmetric operator $\dot A$
 with deficiency indices $(1,1)$. Assume, in addition, that
$A$ is a  self-adjoint extension of $\dot A$. Suppose,  that the
deficiency elements $g_\pm\in \Ker (\dA^*\mp iI)$ are
normalized, $\|g_\pm\|=1$, and chosen in such a way that
\begin{equation}\label{ddoomm14-1}g_++ g_-\in \dom ( A)\,\,\,\text{and}\,\,\,
g_+-\kappa g_-\in \dom (\whA )\,\,\,\text{for some }
\,\,\,|\kappa|<1.
\end{equation}
\end{hypothesis}
\noindent

\begin{remark}\label{r-12}
Without loss of generality, in what follows we assume that $\kappa$ is real and $0\le\kappa<1$: if $\kappa=|\kappa|e^{i\theta}$,
change (the basis) $g_-$ to $e^{i\theta}g_-$ in the deficiency subspace  $\Ker (\dA^*+ i I)$. Thus, for the remainder of this paper we  assume that $\kappa$ is real and $0\le\kappa<1$.
\end{remark}
Both functions $s(z)$ and $S(z)$ can also be defined via \eqref{charsum} and \eqref{ch1} under the condition of Hypothesis \ref{setup-1}. It follows directly from the definition of the function $s(z)$ and \eqref{charsum}  that
$$
s(\dA,A_+)=-s(\dA,A_-),
$$
where $A_+$ and $A_-$ are two self-adjoint extensions of $\dA$ satisfying the first part of Hypotheses \ref{setup} and \ref{setup-1}, respectively. Similar relation takes place between the characteristic functions of the triplets $(\dot A,\whA ,A_+)$ and $(\dot A,\whA ,A_-)$, namely
\begin{equation}\label{e-49-1}
    S(\dA,T ,A_+)=-S(\dA,T ,A_-).
\end{equation}
Here the triplets $(\dot A,\whA ,A_+)$ and $(\dot A,\whA ,A_-)$ satisfy Hypotheses \ref{setup} and \ref{setup-1}, respectively. To prove \eqref{e-49-1} one substitutes $-s(z)$ and $-\kappa$ for $s(z)$ and $\kappa$ in the formula \eqref{ch1} under the assumption that $\kappa$ is real and $0<\kappa<1$.

\section{$(*)$-extensions as state-space operators of L-systems}\label{s3}

In this section we provide an explicit construction of two L-systems based upon given $(*)$-extensions that become state-space operators of obtained systems. We will show that the corresponding operators of these L-systems  satisfy the conditions of Hypotheses \ref{setup} and \ref{setup-1}, respectively.

%described here have  direct connections with generalized Donoghue classes that will become apparent in Section \ref{s5}.

Let $\dA$ be a densely defined closed symmetric operator with finite deficiency indices $(n,n)$.
In this case all  operators $T$ from the class $\Lambda(\dA)$ with $-i\in\rho(T)$
are of the form (see \cite[Theorem 4.1.9]{ABT}, \cite{TSh1})\footnote{For the convenience of the reader and to ease some further calculations we choose to have parameterizing operators $U$ and $\calK$ to carry the opposite signs in von Neumann's formulas \eqref{DOMHAT} and \eqref{e4-60}. As a result, formulas \eqref{e4-62}, \eqref{e3-39-new},  and \eqref{e-H1} slightly deviate from their versions that appear in \cite{BMkT}.}
\begin{equation}\label{e4-60}
    \begin{aligned}
\dom(T)&=\dom(\dA)\oplus(I-\calK )\sN_{i},\quad T=\dA^*\uphar\dom(T),\\
\dom(T^*)&=\dom(\dA)\oplus(I-\calK^*)\sN_{-i},\quad
T^*=\dA^*\uphar\dom(T^*),
    \end{aligned}
\end{equation}
where $\calK\in[\sN_i,\sN_{-i}]$. Let $\calM=\sN_i\oplus\sN_{-i}$ and $P^+$ be a $(+)$-orthogonal projection operator onto a corresponding subspace shown in its subscript.
 Then (see \cite{TSh1}, \cite{BMkT}) all quasi-self-adjoint bi-extensions of $T$ can be parameterized via  an  operator $H\in[\sN_{-i},\sN_i]$ as follows
\begin{equation}\label{e4-61}
  \bA=\dA^*+\calR^{-1} S_\bA P_\calM^+,\quad
\bA^*=\dA^*+\calR^{-1} S_{\bA^*} P_\calM^+,
  \end{equation}
where the block-operator matrices $S_{\bA}$ and $S_{\bA^*}$ are  of the form
\begin{equation}\label{e4-62}
\begin{aligned}
    S_\bA&=\left(
            \begin{array}{cc}
              H\calK & H \\
              \calK(H\calK+iI) & iI+\calK H \\
            \end{array}
          \right),\\
S_{\bA^*}&=\left(
            \begin{array}{cc}
              \calK^*H^*-iI & (\calK^*H^*-iI)\calK^* \\
              H^* & H^*\calK^* \\
            \end{array}
          \right).
\end{aligned}
\end{equation}
Let $T\in\Lambda(\dA)$, $-i\in\rho(T)$ and $A$ be a self-adjoint extension of $\dA$ such that $U$ defines $\dom(A)$ via \eqref{DOMHAT} and $\calK$ defines $T$ via \eqref{e4-60}. It was shown in \cite[Proposition 3]{BMkT} that  $\bA$ is a $(*)$-extension of $T$ whose real part $\RE\bA$ has the quasi-kernel $A$ if and only if $U\calK^*-I$ is invertible and the operator $H$ in \eqref{e4-62} takes the form
\begin{equation}\label{e3-39-new}
H=-i(I-\calK^*\calK)^{-1}[(I-\calK^*U)(I-U^*\calK)^{-1}-\calK^*U]U^*.
\end{equation}

For the remainder of this paper we will be interested in the case when the deficiency indices of $\dA$ are $(1,1)$. Then $S_{\bA}$ and $S_{\bA^*}$  in \eqref{e4-62} become $(2\times2)$ matrices with scalar entries and
As it was announced in \cite{T69} any quasi-self-adjoint bi-extension $\bA$ of $T$ takes a form
\begin{equation}\label{e3-40}
    \bA=\dA^*+\left[p(\cdot,\varphi)+q(\cdot,\psi) \right]\varphi     +\left[v(\cdot,\varphi)+w(\cdot,\psi) \right]\psi,
\end{equation}
where $S_{\bA}=\left(\begin{array}{cc}
                                        p & q \\
                                        v & w \\
                                      \end{array}
                                    \right)$
is a $(2\times2)$ matrix with scalar entries such that $p=H\calK$, $q=H$, $v=\calK(H\calK+i)$, and $w=i+\calK H$. Also, $\varphi$ and $\psi$ are $(-)$-normalized elements in $\calR^{-1}(\sN_i)$ and  $\calR^{-1}(\sN_{-i})$, respectively. Both parameters $H$ and $\calK$ are complex numbers in this case and $|\calK|<1$. Similarly we write
\begin{equation}\label{e-21-star}
    \bA^*=\dA^*+\left[p^\times(\cdot,\varphi)+q^\times(\cdot,\psi) \right]\varphi
    +\left[v^\times(\cdot,\varphi)+w^\times(\cdot,\psi) \right]\psi,
\end{equation}
where $S_{\bA^*}=\left(\begin{array}{cc}
                                        p^\times & q^\times \\
                                        v^\times & w^\times \\
                                      \end{array}
                                    \right)$
is  such that $p^\times=\bar\calK\bar H-i$, $q^\times=(\bar\calK\bar H-i)\bar\calK$, $v^\times=\bar H$, and $w^\times=\bar H\bar\calK$.

Following \cite{BMkT} let us assume\footnote{Throughout this paper  $\kappa$ will be
called the von Neumann extension parameter.} first that $\calK=\calK^*=\bar \calK=\kappa$ and $U=1$. We are going to use \eqref{e3-40} and \eqref{e-21-star} to describe a $(*)$-extension $\bA_1$ parameterized by these values of $\calK$ and $U$. The formula \eqref{e3-39-new} then becomes (see \cite{BMkT})
\begin{equation}\label{e-H1}
H=\frac{-i}{1-\kappa^2}[(1-\kappa)(1-\kappa)^{-1}-\kappa]=\frac{-i}{1+\kappa}
\end{equation}
and
\begin{equation}\label{e-15}
    S_{\bA_1}=\left(
            \begin{array}{cc}
              -\frac{i\kappa}{1+\kappa} & -\frac{i}{1+\kappa} \\
              -\frac{i\kappa^2}{1+\kappa}+i\kappa & i-\frac{i\kappa}{1+\kappa} \\
            \end{array}
          \right),\quad
S_{\bA^*_1}=\left(
            \begin{array}{cc}
              \frac{i\kappa}{1+\kappa}-i & \frac{i\kappa^2}{1+\kappa}-i\kappa  \\
              \frac{i}{1+\kappa} & \frac{i\kappa}{1+\kappa} \\
            \end{array}
          \right).
\end{equation}
Performing direct calculations gives
\begin{align}
    \frac{1}{2i}(S_{\bA_1}-S_{\bA^*_1})&=\frac{1-\kappa}{2+2\kappa}\left(
            \begin{array}{cc}
              1 & -1 \\
              -1 & 1 \\
            \end{array}
            \right ),
            \label{e-16}
        \\
 \frac{1}{2}(S_{\bA_1}+S_{\bA^*_1})&=\frac{i}{2}\left(
            \begin{array}{cc}
              -1 & -1 \\
              1 & 1 \\
            \end{array}
          \right). \nonumber
\end{align}
Using \eqref{e-16} with \eqref{e3-40} one obtains (see \cite{BMkT})
\begin{equation}\label{e-17}
    \begin{aligned}
    \IM\bA_1&=\frac{1-\kappa}{2+2\kappa}\Big((\cdot,\varphi-\psi)(\varphi- \psi)\Big)=(\cdot,\chi_1)\chi_1,
    \end{aligned}
\end{equation}
where
\begin{equation}\label{e-18}
    \chi_1=\sqrt{\frac{1-\kappa}{2+2\kappa}}\,(\varphi- \psi)=\sqrt{\frac{1-\kappa}{1+\kappa}}\left(\frac{1}{\sqrt2}\,\varphi- \frac{1}{\sqrt2}\,\psi\right).
\end{equation}
Also,
\begin{equation}\label{e-17-real}
    \begin{aligned}
    \RE\bA_1&=\dA^*+\frac{i}{2}(\cdot,\varphi+\psi)(\varphi- \psi).
    \end{aligned}
\end{equation}
As one can see from \eqref{e-17-real}, the domain $\dom(\hat A_1)$ of the quasi-kernel $\hat A_1$ of $\RE\bA_1$ consists of such vectors $f\in\calH_+$ that are orthogonal to $(\varphi+ \psi)$.

Consider a special case when $\kappa=0$. Then the corresponding ($*$)-extension $\bA_{1,0}$ is such that
\begin{equation}\label{e-19}
    \IM\bA_{1,0}=\frac{1}{2}(\cdot,\varphi-\psi)(\varphi- \psi)=(\cdot,\chi_{1,0})\chi_{1,0},
    \end{equation}
where
\begin{equation}\label{e-20}
    \chi_{1,0}=\frac{1}{\sqrt2}\,\left(\varphi-\psi\right).
\end{equation}
The $(*)$-extension $\bA_1$ (or $\bA_{1,0}$) that we have just described for the case of $\calK=\kappa$ and $U=1$ can be included in an L-system
\begin{equation}\label{e-62-1-1}
\Theta_1= \begin{pmatrix} \bA_1&K_1&\ 1\cr \calH_+ \subset \calH \subset
\calH_-& &\dC\cr \end{pmatrix}
\end{equation}
with $K_1c=c\cdot\chi_1$, $(c\in\dC)$.

Now let us assume that $\calK=\calK^*=\bar \calK=\kappa$ but $U=-1$ and describe a $(*)$-extension $\bA_2$ parameterized by these values of $\calK$ and $U$.
Then formula \eqref{e3-39-new} yields
\begin{equation}\label{e-H2}
H=\frac{i}{1-\kappa^2}[(1+\kappa)(1+\kappa)^{-1}+\kappa](-1)=\frac{i}{1-\kappa}.
\end{equation}
Similarly to the above, we substitute this value of $H$ into \eqref{e4-62} and obtain
\begin{align}
    S_{\bA_2}&=\left(
            \begin{array}{cc}
              \frac{i\kappa}{1-\kappa} & \frac{i}{1-\kappa} \\
              \frac{i\kappa^2}{1-\kappa}+i\kappa & i+\frac{i\kappa}{1-\kappa} \\
            \end{array}
          \right), \label{e-15-1}
          \\
S_{\bA^*_2}&=\left(
            \begin{array}{cc}
              -\frac{i\kappa}{1-\kappa}-i & -\frac{i\kappa^2}{1-\kappa}-i\kappa  \\
              -\frac{i}{1-\kappa} & -\frac{i\kappa}{1-\kappa} \\
            \end{array}
          \right).\nonumber
\end{align}
Performing direct calculations gives
\begin{align}
    \frac{1}{2i}(S_{\bA_2}-S_{\bA^*_2})&=\frac{1}{2}\cdot\frac{1+\kappa}{1-\kappa}\left(
            \begin{array}{cc}
              1 & 1 \\
              1& 1 \\
            \end{array}
          \right), \label{e-16-1}
          \\
 \frac{1}{2}(S_{\bA_2}+S_{\bA^*_2})&=\frac{i}{2}\left(
            \begin{array}{cc}
              -1 & 1 \\
              -1 & 1 \\
            \end{array}
          \right). \nonumber
\end{align}
Furthermore,
\begin{equation}\label{e-17-1}
    \begin{aligned}
    \IM\bA_2&=\frac{1+\kappa}{2-2\kappa}[(\cdot,\varphi)+(\cdot,\psi)]\varphi+ [(\cdot,\varphi)+(\cdot,\psi)]\psi\\
    &=\frac{1+\kappa}{2-2\kappa}(\cdot,\varphi+\psi)(\varphi+ \psi)\\
    &=(\cdot,\chi_2)\chi_2,
    \end{aligned}
\end{equation}
where
\begin{equation}\label{e-18-1}
    \chi_2=\sqrt{\frac{1+\kappa}{2-2\kappa}}\,(\varphi+ \psi)=\sqrt{\frac{1+\kappa}{1-\kappa}}\left(\frac{1}{\sqrt2}\,\varphi+ \frac{1}{\sqrt2}\,\psi\right).
\end{equation}
Also,
\begin{equation}\label{e-32-real}
    \begin{aligned}
    \RE\bA_2&=\dA^*-\frac{i}{2}(\cdot,\varphi-\psi)(\varphi+\psi).
    \end{aligned}
\end{equation}
As one can see from \eqref{e-32-real}, the domain $\dom(\hat A_2)$ of the quasi-kernel $\hat A_2$ of $\RE\bA_2$ consists of such vectors $f\in\calH_+$ that are orthogonal to $(\varphi- \psi)$.
Similarly to the above when $\kappa=0$, then the corresponding ($*$)-extension $\bA_{2,0}$ is such that
\begin{equation}\label{e-19-1}
    \IM\bA_{2,0}=\frac{1}{2}(\cdot,\varphi+\psi)(\varphi+ \psi)=(\cdot,\chi_{2,0})\chi_{2,0},
    \end{equation}
where
\begin{equation}\label{e-20-1}
    \chi_{2,0}=\frac{1}{\sqrt2}\,\left(\varphi+\psi\right).
\end{equation}
Again we include $\bA_2$ (or $\bA_{2,0}$) into an L-system
\begin{equation}\label{e-62-1-2}
\Theta_2= \begin{pmatrix} \bA_2&K_2&\ 1\cr \calH_+ \subset \calH \subset
\calH_-& &\dC\cr \end{pmatrix}
\end{equation}
with $K_2c=c\cdot\chi_2$, $(c\in\dC)$.

Note that two L-systems $\Theta_1$ and $\Theta_2$ in \eqref{e-62-1-1} and \eqref{e-62-1-2} are constructed in a way that the quasi-kernels $\hat A_1$ of $\RE\bA_1$ and $\hat A_2$ of $\RE\bA_2$ satisfy the conditions of Hypotheses \ref{setup} and \ref{setup-1}, respectively, as it follows from \eqref{e-17-real} and \eqref{e-32-real}.

\section{Impedance functions and Generalized Donoghue classes}\label{s5}

We say (see \cite{BMkT}) that an analytic function $M$ from $\bbC_+$ into itself belongs to the \textbf{generalized Donoghue class} $\sM_\kappa$, ($0\le\kappa<1$) if it admits the representation \eqref{hernev}  where $\mu$ is an infinite Borel measure such that
\begin{equation}\label{e-38-kap}
\int_\bbR\frac{d\mu(\lambda)}{1+\lambda^2}=\frac{1-\kappa}{1+\kappa}\,,\quad\text{equivalently,}\quad M(i)=i\,\frac{1-\kappa}{1+\kappa},
\end{equation}
and to the \textbf{generalized Donoghue class}\footnote{Introducing classes $\sM_\kappa$ and $\sM_\kappa^{-1}$ we followed the notation  of \cite{KK74}.} $\sM_\kappa^{-1}$, ($0\le\kappa<1$)  if it admits the representation \eqref{hernev}
and
\begin{equation}\label{e-39-kap}
\int_\bbR\frac{d\mu(\lambda)}{1+\lambda^2}=\frac{1+\kappa}{1-\kappa}\,,\quad\text{equivalently,}\quad M(i)=i\,\frac{1+\kappa}{1-\kappa}.
\end{equation}
Clearly, $\sM_0=\sM_0^{-1}=\sM$, the (standard) Donoghue class introduced above.

Let
\begin{equation}\label{e-62}
\Theta= \begin{pmatrix} \bA&K&\ 1\cr \calH_+ \subset \calH \subset
\calH_-& &\dC\cr \end{pmatrix}
\end{equation}
be a minimal L-system with one-dimensional input-output space $\dC$ whose main operator $T$ and the quasi-kernel $\hat A$ of $\RE\bA$ satisfy the conditions of Hypothesis \ref{setup}. It is shown in \cite{BMkT} that  the transfer function of $W_\Theta(z)$ and the characteristic function $S(z)$ of the triplet $(\dA, T,\hat A)$ are reciprocals of each other, i.e.,
\begin{equation}\label{e-60}
    W_{\Theta}(z)=\frac{1}{S(z)},\quad z\in\dC_+\cap\rho(T)\cap\rho(\hat A).
\end{equation}
%and $\frac{1}{W_\Theta(z)}\in\sL_\kappa$.
It is also shown in \cite{BMkT} that  the impedance function $V_\Theta(z)$ of $\Theta$ can be represented as
\begin{equation}\label{e-imp-m}
    V_{\Theta}(z)=\left(\frac{1-\kappa}{1+\kappa}\right)V_{\Theta_0}(z),
\end{equation}
where $V_{\Theta_0}(z)$ is  the impedance function of an L-system $\Theta_0$ with the same set of conditions but with $\kappa_0=0$, where $\kappa_0$ is the von Neumann parameter of the main operator $T_0$ of $\Theta_0$.

 Moreover, if $\Theta$ is an arbitrary minimal L-system of the form \eqref{e-62}, then the transfer function of $W_\Theta(z)$ and the characteristic function $S(z)$ of the triplet $(\dA, T,\hat A)$ are related as follows
\begin{equation}\label{e-56-arb}
    W_{\Theta}(z)=\frac{\eta}{S(z)},\quad z\in\dC_+\cap\rho(T)\cap\rho(\hat A),
\end{equation}
where $\eta\in\dC$ and $|\eta|=1$ (see \cite[Corollary 11]{BMkT}).

\begin{lemma}\label{l-7}
Let $\Theta_1$ and $\Theta_2$ be two minimal L-system of the form \eqref{e-62}  whose components satisfy the conditions of Hypothesis \ref{setup} and Hypothesis \ref{setup-1}, respectively.
Then the impedance functions $V_{\Theta_1}(z)$ and $V_{\Theta_2}(z)$ admit the integral representation
\begin{equation}\label{e-60-nu-1}
V_{\Theta_{k}}(z)=\int_\bbR \left(\frac{1}{t-z}-\frac{t}{1+t^2}\right )d\mu_{k}(t),\quad k=1,2.
\end{equation}
\end{lemma}
\begin{proof}
We know (see \cite{ABT}) that both $V_{\Theta_1}(z)$ and $V_{\Theta_2}(z)$ have the integral representation
\begin{equation}\label{e-60-nu-2}
V_{\Theta_{k}}(z)=Q_{k}+\int_\bbR \left(\frac{1}{t-z}-\frac{t}{1+t^2}\right )d\mu_{k}(t),\quad k=1,2,
\end{equation}
and hence our goal is to show that $Q_1=Q_2=0$.
It was shown in \cite[Theorem 12]{BMkT} that under the conditions of Hypothesis \ref{setup} the function $V_{\Theta_1}(z)$ belongs to the class $\sM_\kappa$ and thus $Q_1=0$. We are going to show that the same property takes place for the function $V_{\Theta_2}(z)$. Consider the system $\Theta_2$ satisfying  Hypothesis \ref{setup-1}. In Section \ref{s3} we showed that if an L-system satisfies Hypothesis \ref{setup-1} and hence the parameter $U=-1$ in von Neumann's representation \eqref{DOMHAT} of the quasi-kernel $\hat A_2$ of $\RE\bA_2$, then  $\IM\bA_2=(\cdot,\chi)\chi$, where the vector $\chi$ is given by \eqref{e-18-1}. We know (see \cite{ABT}, \cite{BMkT}) that
\begin{equation}\label{e-52-1}
     V_{\Theta_2}(z)=((\RE\bA_2-z I)^{-1}\chi,\chi)=((\hat A_2-z I)^{-1}\chi,\chi),
\end{equation}
$$
\chi=\frac{1}{\sqrt2}\sqrt{\frac{1+\kappa}{1-\kappa}}\left(\,\varphi+ \psi\right),
$$
where $\hat A_2$ is the quasi-kernel of $\RE\bA_2$ of $\Theta_2$ and $\varphi$, $\psi$ are vectors in $\calH_-$ described in Section \ref{s3}. According to \cite[Theorem 11]{BMkT}, the impedance function of $\Theta_2$ belongs to the Donoghue class $\sM$ if and only if $\kappa=0$. Thus, if we set $\kappa=0$ in \eqref{e-52-1}, then $V_{\Theta_2}(z)\in\sM$ and has $Q_2=0$ in \eqref{e-60-nu-1}. But since the quasi-kernel $\hat A_2$ does not depend on $\kappa$, the expression for $V_{\Theta_2}(z)$ when $\kappa\ne0$ is only different from the case $\kappa=0$ by the constant factor $\frac{1+\kappa}{1-\kappa}$. Therefore $Q_2=0$ in \eqref{e-60-nu-1} no matter what value of $\kappa$ is used.
\end{proof}

Now let us consider a minimal L-system $\Theta$ of the form \eqref{e-62} that satisfies Hypothesis \ref{setup}. Let also
\begin{equation}\label{e-62-alpha}
\Theta_\alpha= \begin{pmatrix} \bA_\alpha&K_\alpha&\ 1\cr \calH_+ \subset \calH \subset
\calH_-& &\dC\cr \end{pmatrix},\quad \alpha\in[0,\pi),
\end{equation}
 be a one parametric family of L-systems such that
 \begin{equation}\label{e-63-alpha}
    W_{\Theta_\alpha}(z)=W_\Theta(z)\cdot (-e^{2i\alpha}),\quad \alpha\in[0,\pi).
 \end{equation}
The existence and structure of $\Theta_\alpha$ were described in details in \cite[Section 8.3]{ABT}. In particular, it was shown that $\Theta$ and $\Theta_\alpha$ share the same main operator $T$ and that
\begin{equation}\label{e-64-alpha}
    V_{\Theta_\alpha}(z)=\frac{\cos\alpha+(\sin\alpha) V_\Theta(z)}{\sin\alpha-(\cos\alpha) V_\Theta(z)}.
\end{equation}
The following theorem takes place.
\begin{theorem}\label{t-6ha}
Let $\Theta$ be a minimal L-system  of the form \eqref{e-62} that satisfies Hypothesis \ref{setup}.
Let also $\Theta_{\alpha}$ be a one parametric family of L-systems given by \eqref{e-62-alpha}-\eqref{e-63-alpha}. Then the impedance function $V_{\Theta_{\alpha}}(z)$ has an integral representation
$$
V_{\Theta_{\alpha}}(z)=\int_\bbR \left(\frac{1}{t-z}-\frac{t}{1+t^2}\right )d\mu_{\alpha}(t)
$$
if and only if $\alpha=0$ or $\alpha=\pi/2$.
\end{theorem}
\begin{proof}
It was proved in \cite[Theorem 12]{BMkT} that $V_\Theta(z)$ belongs to the generalized Donoghue class $\sM_\kappa$  holds if and only if the quasi-kernel $\hat A$ of $\RE\bA$ satisfies the conditions of Hypothesis \ref{setup} which is true in our case. Thus,
 $$
V_\Theta(i)=i\,\frac{1-\kappa}{1+\kappa}.
 $$
Let us set
$$
\Delta=\frac{1-\kappa}{1+\kappa}.
$$
Clearly, $\Delta>0$ since we operate under assumption that $0<\kappa<1$. It follows from \eqref{e-64-alpha} that
$$
\begin{aligned}
V_{\Theta_\alpha}(i)&=\frac{\cos\alpha+(\sin\alpha) V_\Theta(i)}{\sin\alpha-(\cos\alpha) V_\Theta(i)}=\frac{\cos\alpha+i(\sin\alpha) \Delta}{\sin\alpha-i(\cos\alpha) \Delta}\\
&=\frac{(\cos\alpha+i(\sin\alpha) \Delta)(\sin\alpha+i(\cos\alpha) \Delta)}{(\sin\alpha-i(\cos\alpha) \Delta)(\sin\alpha+i(\cos\alpha) \Delta)}=\frac{\cos\alpha\sin\alpha(1-\Delta^2)+i\Delta}{\sin^2\alpha+\cos^2\alpha\cdot\Delta^2}\\
&=\frac{\cos\alpha\sin\alpha(1-\Delta^2)}{\sin^2\alpha+\cos^2\alpha\cdot\Delta^2}+i\frac{\Delta}{\sin^2\alpha+\cos^2\alpha\cdot\Delta^2}.
\end{aligned}
$$
On the other hand, $V_{\Theta_\alpha}(z)$ admits the integral representation of the form \eqref{e-60-nu-1} and hence
$$
\begin{aligned}
V_{\Theta_{\alpha}}(i)&=Q_{\alpha}+\int_\bbR \left(\frac{1}{t-i}-\frac{t}{1+t^2}\right )d\mu_{\alpha}(t)\\
&=Q_{\alpha}+\int_\bbR \left(\frac{1}{t-i}-\frac{t}{(t-i)(t+i)}\right )d\mu_{\alpha}(t)\\
&=Q_{\alpha}+\int_\bbR \frac{1}{t-i}\left(1-\frac{t}{t+i}\right )d\mu_{\alpha}(t)=Q_{\alpha}+i\int_\bbR \frac{1}{1+t^2}\,d\mu_{\alpha}(t).
\end{aligned}
$$
Comparing the two representations of $V_{\Theta_{\alpha}}(i)$ we realize that
\begin{equation}\label{e-65-alpha}
Q_{\alpha}=\frac{\cos\alpha\sin\alpha(1-\Delta^2)}{\sin^2\alpha+\cos^2\alpha\cdot\Delta^2}.
\end{equation}
Analyzing \eqref{e-65-alpha} yields that $Q_{\alpha}=0$ for $\kappa\ne0$ only if either $\alpha=0$ or $\alpha=\pi/2$. Consequently, the only two options for $Q_\alpha$ to be zero are either
$ W_{\Theta_{\pi/2}}(z)=W_\Theta(z)$ or $ W_{\Theta_{0}}(z)=-W_\Theta(z)$. In the former case $\Theta_{\pi/2}=\Theta$ and hence $\Theta_{\pi/2}$ satisfies Hypothesis \ref{setup}. Taking into account Lemma \ref{l-7} we conclude that $\Theta_{0}$ must comply with the conditions of Hypothesis \ref{setup-1} as the only available option.
\end{proof}

The next result describes the relationship between two L-systems of the form \eqref{e-62} that comply with  Hypotheses  \ref{setup}  and  \ref{setup-1}.
\begin{theorem}\label{t-7ha}
Let
\begin{equation}\label{e-62-1}
\Theta_1= \begin{pmatrix} \bA_1&K_1&\ 1\cr \calH_+ \subset \calH \subset
\calH_-& &\dC\cr \end{pmatrix}
\end{equation}
be a minimal L-system whose main operator $T$ and the quasi-kernel $\hat A_1$ of $\RE\bA_1$ satisfy the conditions of Hypothesis \ref{setup} and let
\begin{equation}\label{e-62-2}
\Theta_2= \begin{pmatrix} \bA_2&K_2&\ 1\cr \calH_+ \subset \calH \subset
\calH_-& &\dC\cr \end{pmatrix}
\end{equation}
be another minimal L-system with the same  operators $\dA$ and $T$ as $\Theta_1$ but with the quasi-kernel $\hat A_2$ of $\RE\bA_2$ that satisfies the conditions of Hypothesis \ref{setup-1}. Then
\begin{equation}\label{e-55-1}
    W_{\Theta_1}(z)=-W_{\Theta_2}(z),\quad z\in\dC_+\cap\rho(T),
\end{equation}
and
\begin{equation}\label{e-56-1}
    V_{\Theta_1}(z)=-\frac{1}{V_{\Theta_2}(z)},\quad z\in\dC_+\cap\rho(T).
\end{equation}
\end{theorem}
\begin{proof}
Using Lemma \ref{l-7} and the reasoning above we conclude that since $V_{\Theta_1}(z)$ and $V_{\Theta_2}(z)$ do not have constant terms in their integral representation and $\Theta_1$ is different from $\Theta_2$, equation \eqref{e-55-1} must take place. Then we use relations \eqref{e6-3-6} to obtain \eqref{e-56-1}.
\end{proof}

Next, we present a criterion for the impedance function of a system  to belong to    the generalized Donoghue class $\sM^{-1}_\kappa$.
(cf. \cite[Theorem 12]{BMkT} for the   $\sM_\kappa$-membership criterion).

\begin{theorem}\label{t-22}
Let $\Theta$ of the form \eqref{e-62} be a minimal L-system with the main operator $T$ that has the von Neumann parameter $\kappa$, $(0<\kappa<1)$ and impedance function $V_\Theta(z)$ which is not an identical constant in $\dC_+$.
Then  $V_\Theta(z)$ belongs to the generalized Donoghue class $\sM^{-1}_\kappa$  if and only if the quasi-kernel $\hat A$ of $\RE\bA$ satisfies the conditions of Hypothesis \ref{setup-1}.
\end{theorem}
\begin{proof}
Suppose $\Theta$ is such that the quasi-kernel $\hat A$ of $\RE\bA$ satisfies the conditions of Hypothesis \ref{setup-1}. Then according to Lemma \ref{l-7} it has an integral representation
\begin{equation}\label{e-63-+1}
V_{\Theta}(z)=\int_\bbR \left(\frac{1}{t-z}-\frac{t}{1+t^2}\right )d\mu(t).
\end{equation}
Also, if $\Theta'$ is another system of the form \eqref{e-62} with the same main operator that satisfies the conditions of Hypothesis \ref{setup}, then $V_{\Theta'}(z)\in\sM_\kappa$. Then  we can utilize \eqref{e-56-1} to get
\begin{equation}\label{e-63-+2}
V_{\Theta}(i)=\int_\bbR \frac{1}{1+t^2}\,d\mu(t)=-\frac{1}{V_{\Theta'}(i)}=-\frac{1}{\frac{1-\kappa}{1+\kappa}i}=\frac{1+\kappa}{1-\kappa}i.
\end{equation}
Therefore, $V_{\Theta}(z)\in\sM_\kappa^{-1}$ and we have shown necessity.

Now let us assume $V_{\Theta}(z)\in\sM_\kappa^{-1}$ and hence satisfies \eqref{e-63-+1} and \eqref{e-63-+2}. Then, according to Theorem \ref{t-6ha}, equation \eqref{e-63-+1} implies that $V_{\Theta}(z)$ should either match the case $\alpha=0$ or $\alpha=\pi/2$. But $\alpha=\pi/2$ would cause $\Theta$ to comply with Hypothesis \ref{setup} which would violate \eqref{e-63-+2} since (see \cite{BMkT}) then $V_{\Theta}(z)\in\sM_\kappa$. Thus, $\alpha=0$ and $\Theta$  complies with Hypothesis \ref{setup-1}.
\end{proof}

\begin{remark}\label{r-10}
Let us consider the case when the condition of $V_{\Theta}(z)$ not being an identical constant in $\dC_+$ is omitted in the statement of Theorem \ref{t-22}.
Suppose $V_{\Theta}(z)$ is an identical constant and $\Theta$ satisfies the conditions of Hypothesis \ref{setup-1}. Let $\Theta'$ be the L-system described in the proof of Theorem \ref{t-22} that satisfies the conditions of Hypothesis \ref{setup}. Then, because of  \eqref{e-56-1}, its impedance function $V_{\Theta'}(z)$ is also an identical constant in $\dC_+$ and hence, as it was shown in \cite[Remark 13]{BMkT}
\begin{equation}\label{e-old-remark}
V_{\Theta'}(z)=i\frac{1-\kappa}{1+\kappa},\quad z\in\dC_+.
\end{equation}
Applying \eqref{e-56-1} combined with \eqref{e-old-remark} we obtain
\begin{equation}\label{e-new-remark}
V_{\Theta}(z)=i\frac{1+\kappa}{1-\kappa},\quad z\in\dC_+,
\end{equation}
and hence $V_{\Theta}(z)\in\sM^{-1}_\kappa$.

Now let us assume that $V_{\Theta}(z)\in\sM^{-1}_\kappa$.
We will show that in this case the L-system $\Theta$ from the statement of Theorem \ref{t-22} is bi-unitarily equivalent to an L-system $\Theta_a$ that satisfies the conditions of Hypothesis \ref{setup-1}.
Let $V_{\Theta}(z)$ from Theorem \ref{t-22} takes a form \eqref{e-new-remark}. Let also $\mu(\lambda)$ be a Borel measure on $\mathbb R$ given by the simple formula
\begin{equation}\label{e-75-mu}
    \mu(\lambda)=\frac{\lambda}{\pi},\quad \lambda\in\mathbb R,
\end{equation}
and let $V_0(z)$ be a function with integral representation \eqref{hernev} with the measure $\mu$, i.e.,
\begin{equation*}\label{e-76-V0}
V_0(z)=\int_\bbR \left
(\frac{1}{\lambda-z}-\frac{\lambda}{1+\lambda^2}\right )
d\mu.
\end{equation*}
Then by direct calculations one immediately finds that $V_0(i)=i$ and that $V_0(z_1)-V_0(z_2)=0$ for any $z_1\ne z_2$ in $\dC_+$. Therefore, $V_0(z)\equiv i$ in $\dC_+$ and hence using \eqref{e-new-remark}   we obtain  \begin{equation}\label{e-73-id-const-1}
V_{\Theta}(z)=i\frac{1+\kappa}{1-\kappa}=\frac{1+\kappa}{1-\kappa}\,V_0(z),\quad z\in\dC_+.
\end{equation}
Let us construct a model triple $(\dot \cB,   \whB ,\cB)$ defined by \eqref{nacha1}--\eqref{nacha3} in the Hilbert space $L^2(\bbR;d\mu)$ using the measure $\mu$ from \eqref{e-75-mu} and our value of $\kappa$ (see Apendix \ref{A1} for details). Using the formula \eqref{e-52-def} for the deficiency elements $g_z(\lambda)$ of $\dot B$  and the definition of $s(\dot B, \cB)(z)$ in \eqref{charsum}  we evaluate that $s(\dot B, \cB)(z)\equiv0$ in $\dC_+$. Then, \eqref{ch1} yields $S(\dot \cB,   \whB ,\cB)(z)\equiv \kappa$ in $\dC_+$. Moreover, applying formulas \eqref{e-46-res-form-1} and \eqref{e-47-res-form-2} to the operator $\whB$ in our triple we obtain
\begin{equation}\label{e-77-resolv}
(\whB - zI )^{-1}=(\cB- zI )^{-1}+i\left(\frac{\kappa-1}{2\kappa}\right)(\cdot\, ,g_{\overline{z}})g_z.
\end{equation}
Consider the  operator $\cB_a$ which another self-adjoint extension of $\dot\cB$ and whose domain is given by the formula
\begin{equation}\label{nacha1-a}
\dom(\cB_a)=\dom (\dot \cB)\dot +\linspan\left\{\,\frac{2(\cdot)}{(\cdot)^2 +1}\right \}.
\end{equation}
The model triple $(\dot \cB,   \whB ,\cB_a)$ is now consistent with Hypothesis \ref{setup-1}. Also, \eqref{e-49-1} yields
$$
S(\dot \cB,   \whB ,\cB_a)(z)=-S(\dot \cB,   \whB ,\cB)(z)\equiv -\kappa.
$$

Let us now follow Step 2 of the proof of \cite[Theorem 7]{BMkT} to construct a model L-system $\Theta_a$  corresponding to our model triple $(\dot \cB,   \whB ,\cB_a)$ (see Appendix \ref{A1} for details). Note, that this L-system $\Theta_a$ is minimal by construction, its main operator $\whB$ has regular points in $\dC_+$ due to \eqref{e-77-resolv}, and, according to \eqref{e-60} and \eqref{e-55-1}, $W_{\Theta_a}(z)\equiv -1/\kappa$. But formulas \eqref{e6-3-6} yield that in the case under consideration  $W_{\Theta}(z)\equiv -1/\kappa$. Therefore $W_{\Theta}(z)=W_{\Theta_a}(z)$ and we can (taking into account the properties of $\Theta_a$ we mentioned) apply the Theorem on bi-unitary equivalence \cite[Theorem 6.6.10]{ABT} for L-systems $\Theta$ and $\Theta_a$. Thus we have successfully constructed an L-system $\Theta_a$ that  is bi-unitarily equivalent to the L-system $\Theta$ and satisfies the conditions of Hypothesis \ref{setup-1}.
\end{remark}
%%%%%%%%%%%
\begin{figure}
  \begin{center}
  \includegraphics[width=70mm]{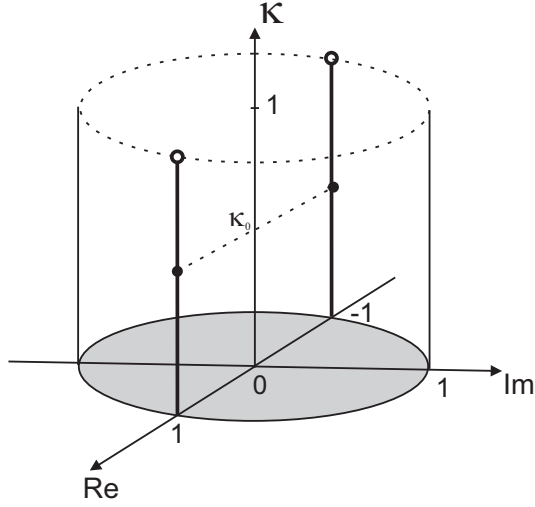}
  \caption{Parametric region $0\le\kappa<1$, $0\le\alpha<\pi$}\label{fig-1}
  \end{center}
\end{figure}
%%%%%%%%%%%
The results of Theorem \ref{t-22} can be illustrated with the help of Figure \ref{fig-1} describing the parametric region for the family of L-systems $\Theta_\alpha$ in \eqref{e-62-alpha}.  When $\kappa=0$ and $\alpha$ changes from $0$ to $\pi$,  every point on the unit circle with cylindrical coordinates $(1,\alpha,0)$, $\alpha\in[0,\pi)$ describes an L-system $\Theta_\alpha$ with $\kappa=0$. Then in this case according to \cite[Theorem 11]{BMkT}  $V_{\Theta_\alpha}(z)$ belongs to the class $\sM$ for any $\alpha\in[0,\pi)$. On the other hand, for any $\kappa_0$ such that $0<\kappa_0<1$ we apply \cite[Theorem 12]{BMkT} to conclude that only the point $(1,0,\kappa_0)$ on the wall of the cylinder is responsible for an L-system  $\Theta_{0}$ such that $V_{\Theta_{0}}(z)$ belongs to the class $\sM_{\kappa_0}$. Similarly, Theorem \ref{t-22} yields that only the point $(-1,0,\kappa_0)$ corresponds to an L-system  $\Theta_{\pi/2}$ such that $V_{\Theta_{\pi/2}}(z)$ belongs to the class $\sM_{\kappa_0}^{-1}$.

We should also note that Theorem \ref{t-6ha} and the reasoning above imply that for the family of L-systems $\Theta_\alpha$ with a fixed value of $\kappa$, ($0<\kappa<1$) \textit{the only two} family members $\Theta_{0}$ and $\Theta_{\pi/2}$ are such that the impedance functions $V_{\Theta_{0}}(z)$ and $V_{\Theta_{\pi/2}}(z)$ belong to the classes $\sM_{\kappa}$  and $\sM_{\kappa}^{-1}$, respectively.

\begin{theorem}\label{t-14}
Let  $V(z)$ belong to the generalized Donoghue class $\sM_\kappa^{-1}$, $0\le\kappa<1$. Then $V(z)$ can be realized as the impedance function $V_{\Theta_\kappa}(z)$ of an L-system $\Theta_\kappa$
 of the form \eqref{e-62} with  the triple $(\dot A, T, \hat A)$ that satisfies Hypothesis \ref{setup-1} with $A=\hat A$, the quasi-kernel of $\RE\bA$. Moreover,
 \begin{equation}\label{e-73-new}
    V(z)=V_{\Theta_\kappa}(z)= \frac{1+\kappa}{1-\kappa}\,M(\dA, \hat A)(z),\quad z\in\dC_+,
 \end{equation}
 where $M(\dA, \hat A)(z)$ is the Weyl-Titchmarsh function associated with the pair $(\dA, \hat A)$.
\end{theorem}
\begin{proof}
It is easy to see that since $V(z)\in\sM_\kappa^{-1}$, then  $V_1(z)=-(1/V(z))\in\sM_\kappa$. Indeed, by direct check $V_1(z)$ is a Herglotz-Nevanlinna function and condition \eqref{e-39-kap} for $V(z)$ immediately implies condition \eqref{e-38-kap} written for $V_1(z)$, i.e., $V_1(i)=i$. Hence, $V_1(z)\in\sM_\kappa$. Consequently, according to \cite[Theorem 13]{BMkT}, $V_1(z)$ can be realized as the impedance function  of an L-system $\Theta'_\kappa$  of the form \eqref{e-62} with  the triple $(\dot A, T, \hat A_1)$ that satisfies Hypothesis \ref{setup} with $A=\hat A_1$, the quasi-kernel of $\RE\bA$. Moreover,
 \begin{equation}\label{e-73-old}
    V_1(z)= \frac{1-\kappa}{1+\kappa}\,M(\dA, \hat A_1)(z),\quad z\in\dC_+,
 \end{equation}
 where $M(\dA, \hat A_1)(z)$ is the Weyl-Titchmarsh function associated with the pair $(\dA, \hat A_1)$. Let $\Theta_\kappa$ be another L-system with the same  operators $\dA$ and $T$ as $\Theta'_\kappa$ but with the quasi-kernel $\hat A$ of $\RE\bA$ that satisfies the conditions of Hypothesis \ref{setup-1}. Then Theorem \ref{t-7ha} and \eqref{e-56-1} yield
$$
V_{\Theta_\kappa}(z)=-\frac{1}{V_{\Theta'_\kappa}(z)}=-\frac{1}{V_{1}(z)}=V(z).
$$
Therefore, $V(z)$ is realized by $\Theta_\kappa$ as required.
In order to show \eqref{e-73-new} we use \eqref{e-73-old} together with the fact that $M(\dA, \hat A)(z)=-1/M(\dA, \hat A_1)(z)$. The latter follows from the formula
\begin{equation*}\label{transm}
M(\dA, \hat A_1)=\frac{\cos \alpha \, M(\dot A, \hat A)-\sin \alpha}{
\cos\alpha +\sin \alpha \,M(\dot A, \hat A)},
\quad\alpha \in [0,\pi),
\end{equation*}
(see \cite[Subsection 2.2]{MT10}) applied for $\alpha=\pi/2$.
\end{proof}

As a closing remark to this section we will make one important observation. Let $\Theta_1$ and $\Theta_2$ be two minimal L-systems given by \eqref{e-62-1} and \eqref{e-62-2} described in the statement of Theorem \ref{t-7ha}. Both L-systems share the same main operator $T$ with the von Neumann parameter $\kappa$ while $\Theta_1$ satisfies the conditions of Hypothesis \ref{setup} and $\Theta_2$ satisfies the conditions of Hypothesis \ref{setup-1}. As we have shown above in this case the impedance functions  $V_{\Theta_1}(z)$ and $V_{\Theta_2}(z)$ belong to the classes $\sM_\kappa$ and $\sM_\kappa^{-1}$, respectively. Consequently, $V_{\Theta_1}(z)$ and $V_{\Theta_2}(z)$ admit integral representations \eqref{e-60-nu-1} with the measures $\mu_1(\lambda)$ and $\mu_2(\lambda)$, respectively. Thus,
\begin{equation}\label{e-73-remark}
   \int_\bbR\frac{d\mu_1(\lambda)}{1+\lambda^2}=\frac{1-\kappa}{1+\kappa}<1,\quad\textrm{ and }\quad  \int_\bbR\frac{d\mu_2(\lambda)}{1+\lambda^2}=\frac{1+\kappa}{1-\kappa}>1,
\end{equation}
for the same fixed value of $\kappa$. Setting,
$$
L_1=\int_\bbR\frac{d\mu_1(\lambda)}{1+\lambda^2}\quad\textrm{ and }\quad L_2=\int_\bbR\frac{d\mu_1(\lambda)}{1+\lambda^2},
$$
and solving both parts of \eqref{e-73-remark} for $\kappa$ we obtain
$$
\kappa=\frac{1-L_1}{1+L_1}=\frac{L_2-1}{L_2+1}.
$$
It also follows from \eqref{e-73-remark} that $L_1L_2=1$ or
$$
\int_\bbR\frac{d\mu_1(\lambda)}{1+\lambda^2}=1/\int_\bbR\frac{d\mu_2(\lambda)}{1+\lambda^2}.
$$
%%%%%%%%%%%%%%%%%%%%

\section{The coupling of two L-systems}\label{s6}

In this Section we will heavily rely on the concept of coupling of two quasi-self-adjoint dissipative unbounded operators introduced in \cite{MT10}. Let us denote the set of all such  dissipative unbounded operators satisfying Hypothesis \ref{setup} by $\sD$.
\begin{definition}\label{op-coup}
Suppose   that
  $T_1\in \sD(\cH_1)$   and $T_2\in \sD(\cH_2)$ are maximal dissipative unbounded operators acting in the Hilbert spaces $\cH_1$ and $\cH_2$, respectively.
We say that the maximal dissipative  operator  $T\in \sD(\cH_1\oplus\cH_2)$ is an operator coupling   of $T_1$ and $T_2$, in writing, $$T=T_1\uplus T_2,$$
 if the Hilbert space $\cH_1$ is invariant for $T$, i.e.,
 $$
   \dom(T)\cap \cH_1=\dom (T_1),\quad   T\big|_{\dom(T_1)}=T_1,
$$
   and  the symmetric operator  $\dot A=   T|_{\dom(T)\cap \dom (T^*)}$ has the property
$$
\dot A\subset T_1\oplus T_2^*.
 $$
\end{definition}

 %\subsection*{Remark}
 The following procedure  provides an explicit algorithm for constructing an operator coupling of two dissipative operators $T_1$ and $T_2$.
Introduce the notation,
$$
\dot A_k=T_k|_{\dom(T_k)\cap\dom(T_k^*)},\quad k=1,2,
$$and fix a basis
$g_\pm^k \in (\Ker (\dA_k^*\mp iI)$, $\|g_\pm\|=1$, $k=1,2$, in the corresponding deficiency subspaces
such that
$$
g_+^k-\kappa_kg_-^k\in \dom (T_k),\quad \text{with} \quad 0\le \kappa_k<1.
$$
It was shown in \cite[Lemma 5.3]{MT10} that there exists a one parameter family $[0, 2\pi)\ni \theta\mapsto \dot A_\theta$ of symmetric restrictions   with deficiency indices $(1,1)$ of the operator  $(\dot A_1\oplus\dot A_2)^*$ such that
$$
\dot A_1\oplus \dot A_2\subset \dot A_\theta\subset T_1\oplus T_2^*,\quad \theta\in [0,2\pi).
 $$
Without loss of generality we can pick the case when $\theta=0$ and set $\dA=\dA_0$. Then by \cite[Lemma 5.3]{MT10}
the domain of $\dot A$ admits the representation
$$\dom (\dot A)=\dom(\dot A_1\oplus \dot A_2)\dot +\cL_0,
$$
where
\begin{equation}\label{sravnieche}
\cL_0=
\begin{cases}
\linspan\left \{ \sin \alpha \,\, (g_+^1-\kappa_1 g_-^1)
\oplus\cos \alpha \left (
g_+^2- \frac{1}{\kappa_2}g_-^2 \right )\right \},& \kappa_2\ne 0
\\
\linspan\left  \{   (g_+^1-\kappa_1 g_-^1)
\oplus
\left (-\sqrt{1-\kappa_1^2}g_-^2 \right )\right \},&   \kappa_2=0
\end{cases}
\end{equation}
and
\begin{equation}\label{tana11}
\tan \alpha=\frac{1}{\kappa_2}\sqrt{\frac{1-\kappa_2^2}{1-\kappa_1^2}}, \quad \kappa_2\ne 0.
\end{equation}
In accordance with \cite[Theorem 5.4]{MT10}, we introduce the maximal dissipative extension $T$ of $\dot A$
defined as the restriction of $(\dot A_1\oplus\dot A_2)^*$ on
\begin{equation}\label{e-68-new}
\dom (T)=\dom (\dot A)\dot +\linspan \left \{ G_+-\kappa_1\kappa_2 G_-\right \},
\end{equation}
where the deficiency elements $G_\pm$ of $\dot A$ are given by
\begin{equation}\label{GGG}
G_+=  \cos \alpha \, g_+^1-\sin  \alpha \,g_+^2,
\end{equation}
$$
G_-= \cos  \, \beta  g_-^1- \sin \beta \,g^2_-.
$$
Here
$$\sin \beta =\kappa_1\sin\alpha \quad \text{and }\quad \cos \beta=\frac{1}{\kappa_2}\cos \alpha,
\quad \text{if }  \kappa_2\ne 0,
$$ and
$$\sin \beta =\kappa_1 \quad \text{and }\quad \cos \beta=\sqrt{1-\kappa_1^2},\quad  \text{if }  \kappa_2= 0.$$
By construction,
\begin{equation}\label{222}
\dot A=T|_{\dom (T_1)\cap \dom(T_2^*)},
\end{equation}
and the operator $T$ is a unique coupling of $T_1$ and $T_2$.

We should also mention that the self-adjoint extension $A$ of $\dA$ is also uniquely defined by the requirements
\begin{equation}\label{refop}
G_+-G_-\in \dom(A),\quad A=\dA^*\Big|_{\dom(A)}.
\end{equation}

The following theorem is proved in \cite{MT10}.
\begin{theorem}\label{opcoup}
Suppose   that $T=T_1\uplus T_2$ is an operator coupling  of maximal dissipative operators
  $T_k \in \sD(\cH_k)$, $k=1,2$. Denote by  $\dot A $, $\dot A_1$ and $\dot A_2$  the corresponding underlying symmetric operators with deficiency indices $(1,1)$, respectively.
That is,  $$\dot A=T|_{\dom(T)\cap\dom(T^*)}
$$
and
$$
\dot A_k=T_k|_{\dom(T_k)\cap\dom(T_k^*)}\quad k=1,2.
$$
  Then there exist self-adjoint reference operators  $A$, $A_1$,and $A_2$,   extensions of
  $\dot A$, $\dot A_1$ and $\dot A_2$, respectively,  such that
\begin{equation}\label{proizvas}S(T_1\uplus T_2,  A)=S(T_1,A_1)S(T_2, A_2).
\end{equation}
\end{theorem}
One important corollary to Theorem \ref{opcoup} was presented in \cite{MT10} and can be summarized as
  the following multiplication formula for the von Neumann parameters
\begin{equation}\label{mult-for}
\kappa({T_1\uplus T_2})=\kappa({T_1}) \cdot \kappa({T_2}).
\end{equation}
In this formula each $\kappa$ represents the corresponding von Neumann parameter for the operator in parentheses.

Our goal is to give a definition of a coupling of two L-systems of the form \eqref{e-62}.
Let us explain how all elements of this coupling are constructed. The main operator $T$ of the system $\Theta$ is defined to be  the coupling of main operators $T_1$ and $T_2$ of L-systems $\Theta_1$ and $\Theta_2$. That is
\begin{equation}\label{e-70-cup}
    T=T_1\uplus T_2,
\end{equation}
where operator coupling is defined by Definition \ref{op-coup} and its construction is described earlier in this section. The symmetric operator $\dA$ of the system $\Theta$ is naturally given by
$$
\dot A=T|_{\dom(T)\cap\dom( T^*)}.
$$
The rigged Hilbert space $\calH_{+} \subset \calH_1\oplus\calH_2 \subset\calH_{-}$ is constructed as the operator based rigging (see \cite[Section 2.2]{ABT}) such that $\calH_+=\dom(\dA^*)$. A self-adjoint extension $\hat A$ of $\dA$ is the quasi-kernel of $\RE\bA$ and is defined by \eqref{refop}. The $(*)$-extension $\bA$ of the system is defined uniquely based on operators $T$ and $\hat A$ and is described by formula \eqref{e3-40}. Then using formulas \eqref{e-17} and \eqref{e-18}  one gets
\begin{equation}\label{e-73-un}
       \IM\bA=(\cdot,\chi)\,\chi,\quad \chi=\sqrt{\frac{1-\kappa}{1+\kappa}}\left(\frac{1}{\sqrt2}\,\Phi- \frac{1}{\sqrt2}\,\Psi\right).
    \end{equation}
Here, $\kappa=\kappa_1\cdot\kappa_2$ (see \eqref{mult-for}) is the von Neumann parameter of the operator $T$ in \eqref{e-70-cup} while $\kappa_1$ and $\kappa_2$ are corresponding von Neumann parameters of $T_1$ and $T_2$ of $\Theta_1$ and $\Theta_2$, respectively. Also, $\Phi=\calR^{-1}G_+$ and $\Psi=\calR^{-1}(g_-)$, where $\calR$ is the Riesz-Berezansky   operator for the rigged Hilbert state space of $\Theta$.
The channel operator $K$ of the system is then given by $K c=c\cdot \chi$, $c\in\dC$ and $K^*f=(f,\chi)$, $(f\in\calH_+)$.
Now we are ready to give a definition of a coupling of two L-systems of the form \eqref{e-62}.
\begin{definition}\label{d6-7-2}
Let %$\Theta_1$ and $\Theta_2$
\begin{equation}\label{e6-7-31}
\Theta_i =\begin{pmatrix}
\bA_i &K_i &1\\
\calH_{+i}\subset  \calH_i\subset  \calH_{-i} &{ } &\dC
\end{pmatrix}, \quad i=1,2,
\end{equation}
be two minimal L-systems of the form \eqref{e-62} that satisfy the conditions of Hypothesis \ref{setup}.
 An L-system
 \begin{equation*}%\label{e-62}
\Theta= \begin{pmatrix} \bA&K&\ 1\cr \calH_{+} \subset \calH_1\oplus\calH_2 \subset
\calH_{-}& &\dC\cr
\end{pmatrix}
\end{equation*}
is called a \textbf{coupling of two L-systems} if the operators $\bA$, $K$ and the rigged space $\calH_{+} \subset \calH_1\oplus\calH_2 \subset
\calH_{-}$ are defined as described above. In this case we will also write
$$
\Theta=\Theta_1\cdot\Theta_2.
$$
\end{definition}
Clearly, by construction, the elements of L-system $\Theta$ satisfy the conditions of Hypothesis \ref{setup} for $\kappa=\kappa_1\cdot\kappa_2$.
\begin{theorem}\label{t6-7-3} Let an L-system $\Theta$ be the coupling of two L-systems $\Theta_1$
and $\Theta_2$ of the form \eqref{e6-7-31} that satisfy the conditions of Hypothesis \ref{setup}. Then if $z\in\rho(T_1)\cap\rho(T_2)$,  we have
\begin{equation}\label{e-91-mult}
W_\Theta(z)=W_{\Theta_1}(z)\cdot W_{\Theta_2}(z).
\end{equation}
\end{theorem}
\begin{proof}
The proof of the Theorem immediately follows from Theorems \ref{opcoup} and formula \eqref{e-60}.
\end{proof}
A different type of L-system coupling with property \eqref{e-91-mult} was constructed in \cite{BT33} (see also \cite[Section 7.3]{ABT}).
\subsection*{Remark 6}\label{r-7} Definition \ref{d6-7-2} for the coupling of two L-systems can be meaningfully extended to the case when participating L-systems $\Theta_1$ and $\Theta_2$ do not satisfy the requirements of Hypothesis \ref{setup}. We use the following procedure to explain the system $\Theta=\Theta_1\cdot\Theta_2$. We begin by introducing L-systems $\Theta_1'$ and $\Theta_2'$ which are only different from $\Theta_1$ and $\Theta_2$ by the fact that their quasi-kernels $\hat A_1'$ of $\RE\bA'_1$ and $\hat A_2'$ of $\RE\bA'_2$ satisfy the conditions of Hypothesis \ref{setup} for the same values of $\kappa_1$ and $\kappa_2$. Then Definition \ref{d6-7-2} applies to them and we can obtain the coupling $\Theta'=\Theta_1'\cdot\Theta_2'$. Note that $\Theta'$ is an L-system that satisfies the conditions of Hypothesis \ref{setup}. Theorem \ref{t6-7-3} also applies and we have
$$
W_{\Theta'}(z)=W_{\Theta_1'}(z)\cdot W_{\Theta_2'}(z)=\nu_1 W_{\Theta_1}(z)\cdot \nu_2 W_{\Theta_2}(z),
$$
where $|\nu_1|=|\nu_2|=1$. Equivalently,
$$
    (\nu_1\nu_2)^{-1}W_{\Theta'}(z)=W_{\Theta_1}(z)\cdot W_{\Theta_2}(z).
$$
We set
\begin{equation}\label{e-73-nu}
W_\Theta(z)=(\nu_1\nu_2)^{-1}W_{\Theta'}(z),
\end{equation}
and recover the L-system $\Theta$ whose transfer function is $W_\Theta(z)$ and that is different from the system $\Theta'$ by changing the quasi-kernel $\hat A'$ to $\hat A$ in a way that produces the unimodular factor $(\nu_1\nu_2)^{-1}$ in \eqref{e-73-nu}. This construction of $\Theta$ is well defined and unique due to the theorems about a constant $J$-unitary factor \cite[Theorem 8.2.1]{ABT} and \cite[Proposition 3]{BMkT}.

\begin{theorem}\label{t28} Let an L-system $\Theta$ be the coupling of two L-systems $\Theta_1$
and $\Theta_2$ of the form \eqref{e6-7-31} that satisfy the conditions of Hypothesis \ref{setup}. Suppose also that $V_{\Theta_1}(z)\in\sM_{\kappa_1}$ and $V_{\Theta_2}(z)\in\sM_{\kappa_2}$. Then  $V_\Theta(z)\in\sM_{\kappa_1\kappa_2}$.
\end{theorem}
\begin{proof}
As we mentioned earlier, the elements of L-system $\Theta$ satisfy the conditions of Hypothesis \ref{setup} for $\kappa=\kappa_1\cdot\kappa_2$. Thus,  according to \cite[Theorem 12]{BMkT}, $V_\Theta(z)$ satisfies \eqref{e-imp-m} and consequently
$$V_\Theta(i)= i\,\frac{1-\kappa}{1+\kappa},$$
and hence $V_\Theta(z)\in\sM_{\kappa_1\kappa_2}$.
\end{proof}
The following statement immediately follows from Theorem \ref{t28}.
\begin{corollary}\label{c-29} Let an L-system $\Theta$ be the coupling of two L-systems $\Theta_0$
and $\Theta_1$ of the form \eqref{e6-7-31} that satisfy the conditions of Hypothesis \ref{setup}.  Suppose also that $V_{\Theta_0}(z)\in\sM$ and $V_{\Theta_2}(z)\in\sM_{\kappa}$. Then  $V_\Theta(z)\in\sM$.
\end{corollary}
The coupling ``absorbtion" property described by Corollary \ref{c-29} can be enhanced as follows.
\begin{corollary}\label{c-30} Let an L-system $\Theta$ be the coupling of two L-systems $\Theta_0$ and $\Theta_1$ of the form \eqref{e6-7-31} and let $V_{\Theta_0}(z)\in\sM$. Then  $V_\Theta(z)\in\sM$.
\end{corollary}
\begin{proof}
The procedure of forming the coupling of two L-systems that don't necessarily obey the conditions of Hypothesis \ref{setup} is described in Remark \ref{r-7}. It is clear that the main operator $T$ of such coupling is also found via \eqref{e-70-cup} and its von Neumann parameter $\kappa$ is still a product given by \eqref{mult-for}. Hence in our case $\kappa=0$ regardless of whether $\Theta_0$ and $\Theta_1$ satisfy the conditions of Hypothesis \ref{setup}. Therefore applying \cite[Theorem 11]{BMkT} yields $V_\Theta(z)\in\sM$.
\end{proof}
\begin{corollary}\label{c-19} Let an L-system $\Theta$ be the coupling of two L-systems $\Theta_1$ and $\Theta_2$ of the form \eqref{e6-7-31} that satisfy the conditions of Hypothesis  \ref{setup-1}. Suppose also that $V_{\Theta_1}(z)\in\sM_{\kappa_1}^{-1}$ and $V_{\Theta_2}(z)\in\sM_{\kappa_2}^{-1}$. Then $\Theta$ satisfies the conditions of Hypothesis \ref{setup} and $V_\Theta(z)\in\sM_{\kappa_1\kappa_2}$.
\end{corollary}
\begin{proof}
In order to create the coupling of two systems $\Theta_1$ and $\Theta_2$ that do not obey the conditions of Hypothesis \ref{setup} we follow the procedure described in Remark \ref{r-7}. We know, however, that according to Theorem \ref{t-7ha} and \eqref{e-55-1} both unimodular factors $\nu_1$ and $\nu_2$ in \eqref{e-73-nu} are such that $\nu_1=\nu_2=-1$. Then \eqref{e-73-nu} yields
$$
W_\Theta(z)=((-1)(-1))^{-1}W_{\Theta'}(z)=W_{\Theta'}(z)=W_{\Theta_1'}(z)\cdot W_{\Theta_2'}(z),
$$
where $W_{\Theta'}(z)$ is the transfer function of the L-system $\Theta'$ that obeys Hypothesis \ref{setup}.

Let us show that in fact the L-systems $\Theta$ and $\Theta'$ coincide.  Recall  (see Remark \ref{r-7}) that $\Theta$ and $\Theta'$ can only be different by the quasi-kernels $\hat A$ and $\hat A'$ and thus share the same operator $T$. Then according to the theorem about a constant $J$-unitary factor \cite[Theorem 8.2.1]{ABT},   $W_{\Theta}(z)=\nu W_{\Theta_1}(z)$, where $|\nu|=1$ is a unimodular complex number. In our case, however, we have just shown that $W_\Theta(z)=W_{\Theta'}(z)$ and hence $\nu=1$. As it can be seen in the proof of \cite[Theorem 8.2.1]{ABT}
then only possibility for  $\nu=1$ occurs when von Neumann's parameters for the quasi-kernels $\hat A$ and $\hat A'$ match (see also \cite[Theorem 4.4.6]{ABT}). But this implies that $\hat A=\hat A'$ and thus $\Theta$ and $\Theta'$ coincide. Consequently, $\Theta$  satisfies the conditions of Hypothesis \ref{setup}.
Also, both $\Theta_1'$ and $\Theta_2'$ share the same main operators $T_1$ and $T_2$ with $\Theta_1$ and $\Theta_2$, respectively. Therefore, according to Theorem \ref{t28},  $V_\Theta(z)=V_{\Theta'}(z)\in\sM_{\kappa_1\kappa_2}$.
\end{proof}
\begin{corollary}\label{c-23} Let an L-system $\Theta$ be the coupling of two L-systems $\Theta_1$ and $\Theta_2$ of the form \eqref{e6-7-31} that both satisfy the conditions of Hypotheses \ref{setup} and  \ref{setup-1}, respectively. Suppose also that $V_{\Theta_1}(z)\in\sM_{\kappa_1}$ and $V_{\Theta_2}(z)\in\sM_{\kappa_2}^{-1}$. Then $\Theta$ satisfies the conditions of Hypothesis \ref{setup-1} and $V_\Theta(z)\in\sM_{\kappa_1\kappa_2}^{-1}$.
\end{corollary}
\begin{proof}
In order to prove the statement of the corollary we replicate all the steps of the proof of Corollary \ref{c-19} above with  $\nu_1=1$ and $\nu_2=-1$ implying
$$
W_\Theta(z)=((1)(-1))^{-1}W_{\Theta'}(z)=-W_{\Theta'}(z)=-W_{\Theta_1'}(z)\cdot W_{\Theta_2'}(z)=W_{\Theta''}(z),
$$
where $W_{\Theta''}(z)$ is the transfer function of the L-system $\Theta''$ that  has the same  operators $\dA$ and $T$ as $\Theta'$ but with the quasi-kernel $\hat A''$ of $\RE\bA''$ that satisfies the conditions of Hypothesis \ref{setup-1}.
Then similar reasoning  is used to show that the L-system $\Theta$ satisfies the conditions of Hypothesis \ref{setup-1}  and that $V_\Theta(z)\in\sM_{\kappa_1\kappa_2}^{-1}$.
\end{proof}
Note that the conclusion of Corollary \ref{c-23} remains valid if $V_{\Theta_1}(z)\in\sM_{\kappa_1}^{-1}$ and $V_{\Theta_2}(z)\in\sM_{\kappa_2}$. Figure \ref{fig-2} above is provided to illustrate Theorems \ref{t6-7-3}--\ref{t28} as well as Corollaries \ref{c-29}--\ref{c-23} including the ``absorbtion property".

%%%%%%%%%%%
\begin{figure}
  \begin{center}
  \includegraphics[width=120mm]{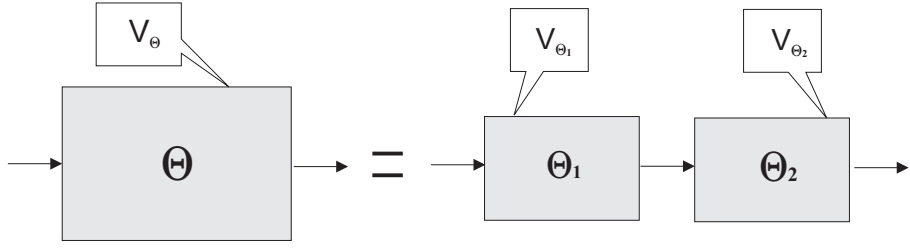}
  \caption{Coupling $\Theta=\Theta_1\cdot\Theta_2$ and its impedance}\label{fig-2}
  \end{center}
\end{figure}
%%%%%%%%%%%
%%%%%%%%%%%%%%

\section{The Limit coupling}\label{s7}

Let $\Theta_1$, $\Theta_2$,\ldots, $\Theta_n$ be  L-systems of the form \eqref{e-62}. Relying on Definition \ref{d6-7-2} and  Remark \ref{r-7} from the previous section we can inductively define the coupling of $n$ L-systems by
\begin{equation}\label{e77-ind}
    \Theta^{(n)}=\prod_{k=1}^n \Theta_k=\Theta_1\cdot\Theta_2\cdot\ldots\cdot \Theta_n,
\end{equation}
with all the elements of the L-system $\Theta^{(n)}$  constructed according to the repeated use of the coupling procedure described in Section \ref{s6}. In particular, if each of the L-systems $\Theta_k$, ($k=1,\ldots,n$) satisfies Hypothesis \ref{setup}, then so does the L-system $\Theta^{(n)}$.

\begin{definition}\label{d17}
Let $\Theta_1,\, \Theta_2,\, \Theta_3,\ldots$  be an infinite sequence of L-systems with von Neumann's parameters $\kappa_1$, $\kappa_2$, $\kappa_3$, \ldots, respectively. We say that a minimal  L-system $\Theta$ is a \textbf{limit coupling } of the sequence $\{\Theta_n\}$ and write
\begin{equation}\label{e78-limit}
  \Theta=\lim_{n\to\infty}  \Theta^{(n)}=\lim_{n\to\infty}\prod_{k=1}^n \Theta_k=\Theta_1\cdot\Theta_2\cdot\ldots\cdot \Theta_n\cdot\ldots,
\end{equation}
if the von Neumann parameter $\kappa$ of $\Theta$ is such that $$\kappa=\lim_{n\to\infty}\prod_{k=1}^n\kappa_n,\quad |\kappa|<1,$$ and
$$
W_\Theta(z)=\lim_{n\to\infty}\prod_{k=1}^n W_{\Theta_n}(z), \quad z\in\dC_-.
$$
\end{definition}
Formula \eqref{e77-ind} and Definition \ref{d17} allow us to introduce the concepts of the ``power" of an L-system and of the {``system attractor"}. Indeed, let $\Theta$ be an L-system of the form \eqref{e-62} satisfying the conditions of Hypothesis \ref{setup}. We refer to the repeated self-coupling
\begin{equation}\label{e79-power}
    \Theta^{n}=\underbrace{\Theta\cdot\Theta\cdot\ldots\cdot \Theta}_{n\textrm{ times}}
\end{equation}
as an $n$-th \textbf{power} of an L-system $\Theta$. Furthermore, the L-system $\Xi$ such that
\begin{equation}\label{e80-limit}
  \Xi=\lim_{n\to\infty}  \Theta^{n},
\end{equation}
will be called a \textbf{system attractor}. Let $\kappa$ and $\kappa_\Xi$ be the von Neumann parameters of the L-systems $\Theta$ and $\Xi$, respectively.  It follows from Definition \ref{d17} and the fact that $|\kappa|<1$ that
$$
\kappa_\Xi=\lim_{n\to\infty}\kappa^n=0.
$$
Moreover, it is known (see \cite[Section 6.3]{ABT}) that $|W_\Theta(z)|<1$ for all $z\in\dC_-$. Hence
\begin{equation}\label{e-81-W}
   W_\Xi(z)=\lim_{n\to\infty} \big(W_{\Theta}(z)\big)^n=0, \quad z\in\dC_-.
\end{equation}
Consequently, for the impedance function $V_\Xi (z)$ relation \eqref{e6-3-6} yields
$$
V_\Xi (z) = i \frac{W_\Xi (z) - 1}{W_\Xi (z) + 1} \equiv-i, \quad z\in\dC_-.
$$
%%%%%%%%%%%
\begin{figure}
  \begin{center}
  \includegraphics[width=120mm]{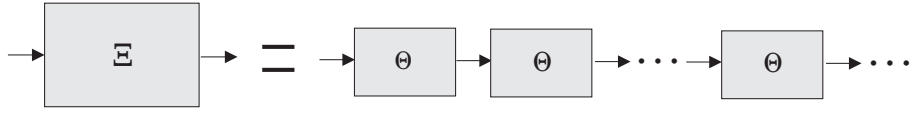}
  \caption{System attractor $\Xi$}\label{fig-3}
  \end{center}
\end{figure}
%%%%%%%%%%%
Since $V_\Xi (z)$ is a Herglotz-Nevanlinna function, then $V_\Xi (z)\equiv i$ for all $z\in\dC_+$. To show that a system attractor  exists we consider the model triple $(\dot\cB,\whBo,\cB)$ of the form \eqref{nacha1}--\eqref{nacha3} with $S(\dot\cB,\whBo,\cB)(i)=0$ (see Apendix \ref{A1} for details) and a model system
\begin{equation}\label{e-82-attract}
   \Xi=\begin{pmatrix} \dB_0&K_0&\ 1\cr \calH_+ \subset \calH \subset
\calH_-& &\dC\cr \end{pmatrix},
\end{equation}
constructed in \cite[Section 5]{BMkT} upon this triple. Recall, that this system $\Xi$ the state-space operator $\dB_0\in[\calH_+,\calH_-]$ is a  $(*)$-extension of $\whBo$ such that $\RE\dB_0\supset \cB=\cB^*$, $K_0^*f=(f,g_0)$, $(f\in\calH_+)$, and $\kappa=0$. It is shown in \cite[Theorem 7]{BMkT} that for the system $\Xi$ we have $V_\Xi(z)=M(\dot\cB,\cB)(z)$, $z\in\dC_+$. Taking into account that $\kappa=0$, we note that resolvent formula \eqref{e-46-res-form-1}--\eqref{e-47-res-form-2} takes the form
\begin{equation}\label{e-85-res-form}
(\whBo - zI )^{-1}=(\cB- zI )^{-1}-\left (M(\dot \cB, \cB)(z)-i\right )^{-1}(\cdot\, ,g_{\overline{z}})g_z,\;
\end{equation}
$$
 z\in\rho(\whBo)\cap \rho(\cB).
$$
In the case when  $M(\dot\cB,\cB)(z)= i$ for all $z\in\dC_+$, formula \eqref{blog} would imply that $s(\dot\cB,\cB)(z)\equiv0$ in the upper half-plane. Then, \eqref{e-85-res-form} (see also \cite[Lemma 5.1]{MT-S}) yields that all the points $z\in\dC_+$ are eigenvalues for $\whBo$. Consequently, the point spectrum of the dissipative operator $\whBo$ fills in the whole open upper half-plane $\bbC_+$.

The following lemma will be a useful tool in understanding the geometric structure of system attractors.
\begin{lemma}\label{l-22}
Let $\Theta_1$ and $\Theta_2$ be two minimal L-systems of the form \eqref{e6-7-31} such that the corresponding main operators $T_1$ and $T_2$ of these systems have the von Neumann parameters $\kappa_1=\kappa_2=0$. Let also $V_{\Theta_1}(z)=V_{\Theta_2}(z)= i$ for all $z\in\dC_+$. Then $\Theta_1$ and $\Theta_2$ are bi-unitary equivalent to each other.
\end{lemma}
\begin{proof}
Let $\dA_k$ and $\hat A_k$, ($k=1,2$) be the corresponding symmetric operators and quasi-kernels of $\RE\bA_k$, ($k=1,2$) in L-systems $\Theta_1$ and $\Theta_2$. Then, as it was explained in Section \ref{s4} and Remark \ref{r-10}, $s(\dA_k,\hat A_k)=S(\dA_k,T_k,\hat A_k)\equiv0$, ($k=1,2$) in $\dC_+$. Hence, (see \cite{MT-S}) the triplets $(\dA_1,T_1,\hat A_1)$ and
$(\dA_2,T_2,\hat A_2)$ are unitarily equivalent to each other, that is there exists an isometric operator $U$ from $\calH_1$ onto $\calH_2$ such that
\begin{equation}\label{e-93-un-eq}
U T_1=T_2 U,\quad U \dA_1=\dA_2 U,\quad U \hat A_1=\hat A_2 U.
\end{equation}
Let $g_{\pm,k}$  be deficiency elements $g_{\pm,k}\in \Ker (\dot A^*_k\mp iI)$, $\|g_{\pm,k}\|=1$, ($k=1,2$). Then
$$
g_{+,1}- e^{2i\alpha}g_{-,1}\in \dom (\hat A_1),\quad g_{+,2}- e^{2i\beta}g_{-,2}\in \dom (\hat A_2),
$$
for some $\alpha,\,\beta\in[0,\pi)$. Also, $g_{+,1}\in\dom(T_1)$ and $g_{+,2}\in\dom(T_2)$. Moreover, \eqref{e-93-un-eq} implies $U g_{+,1}=g_{+,2}\in\dom(T_2)$ and $U e^{2i\alpha}g_{-,1}=e^{2i\beta}g_{-,2}$.

Following \cite[Theorem 6.6.10]{ABT} we introduce $U_+=U|_{\calH_{+1}}$ to be an isometry from $\calH_{+1}$ onto $\calH_{+2}$, and $U_-=(U_+^*)^{-1}$ is an isometry from $\calH_{-1}$ onto $\calH_{-2}$. Performing similar to Section \ref{s3} steps we obtain
\begin{equation*}\label{e-17-real-atr}
    \begin{aligned}
    \RE\bA_1&=\dA^*_1+\frac{i}{2}(\cdot,\varphi_1+e^{2i\alpha}\psi_1)(\varphi_1- e^{2i\alpha}\psi_1),\\
    \RE\bA_2&=\dA^*_2+\frac{i}{2}(\cdot,\varphi_2+e^{2i\beta}\psi_2)(\varphi_2- e^{2i\beta}\psi_2),
    \end{aligned}
\end{equation*}
and
\begin{equation*}\label{e-18-im-atr}
    \begin{aligned}
    \IM\bA_1&=\frac{1}{2}(\cdot,\varphi_1-e^{2i\alpha}\psi_1)(\varphi_1- e^{2i\alpha}\psi_1),\\
    \IM\bA_2&=\frac{1}{2}(\cdot,\varphi_2-e^{2i\beta}\psi_2)(\varphi_2- e^{2i\beta}\psi_2),
    \end{aligned}
\end{equation*}
where $\varphi_k$ and $\psi_k$, ($k=1,2$) are $(-)$-normalized elements in $\calR^{-1}(\Ker (\dot A^*_k- iI))$ and  $\calR^{-1}(\Ker (\dot A^*_k+ iI))$, ($k=1,2$), respectively.
It is not hard to show that $U_-(\varphi_1- e^{2i\alpha}\psi_1)=\varphi_2- e^{2i\beta}\psi_2$. Indeed, for any $f\in\calH_{+2}$ we have
$$
    \begin{aligned}
(f,\varphi_2- e^{2i\beta}\psi_2)&=(f,g_{+,2}- e^{2i\beta}g_{-,2})_{+2}=(f,U_+(g_{+,1}- e^{2i\alpha}g_{-,1}))_{+2}\\
&=(U_+U_+^{-1}f,U_+(g_{+,1}- e^{2i\alpha}g_{-,1}))_{+2}\\
&=(U_+^{-1}f,g_{+,1}- e^{2i\alpha}g_{-,1})_{+1}\\
&=(U_+^{-1}f,\varphi_1- e^{2i\alpha}\psi_1)=(f,(U_+^*)^{-1}(\varphi_1- e^{2i\alpha}\psi_1))\\
&=(f,U_-(\varphi_1- e^{2i\alpha}\psi_1)).
    \end{aligned}
$$
Combining this with the above formulas for $\IM\bA_1$ and $\IM\bA_2$ we obtain that $\IM\bA_2=U_-\IM\bA_1U_+^{-1}$. Taking into account that \eqref{e-93-un-eq} implies $$\dA_2^*=U\dA_1^*U^{-1}=U_-\dA_1^*U_+^{-1}$$ and using similar steps we also get $\RE\bA_2=U_-\RE\bA_1U_+^{-1}$. Thus, $\bA_2=U_-\bA_1U_+^{-1}$.
\end{proof}

Note  that the  L-system $\Xi$ in \eqref{e-82-attract} is minimal since the symmetric operator $\dot\cB$ is prime. Moreover, it follows from Definition \ref{d17} and Lemma \ref{l-22} that a system attractor $\Xi$ is not unique but all system attractors are bi-unitary equivalent to each other.  The L-system $\Xi$ in \eqref{e-82-attract} will be called the  \textbf{position system attractor}.

Now we are going to  construct another model for a system attractor. In the space $\calH=L^2_{\dR}=L^2_{(-\infty,0]}\oplus L^2_{[0,\infty)}$ we consider a prime symmetric operator
\begin{equation}\label{e-87-sym}
\dA x=i\frac{dx}{dt}
\end{equation}
on
$$
\begin{aligned}
\dom(\dA)&=\left\{x(t)=\left(
                       \begin{array}{c}
                         x_1(t) \\
                         x_2(t) \\
                       \end{array}
                     \right)\,\Big|\,x(t) -\text{abs. cont.},\right.\\
                     &\left. x'(t)\in L^2_{\dR},\, x(0-)=x(0+)=0\right\}.\\
\end{aligned}
$$
Its  deficiency vectors  are easy to find
\begin{equation}\label{e-87-def}
g_z=\left(
    \begin{array}{c}
      e^{-izt}  \\
                0 \\
      \end{array}
     \right),\; \IM z>0,\qquad
g_z=\left(
\begin{array}{c}
               0 \\
      e^{-izt} \\
       \end{array}
     \right),\; \IM z<0.
\end{equation}
In particular, for $z=\pm i$ the   (normalized) deficiency vectors  are
\begin{equation}\label{e-88-def}
g_+=\left(
    \begin{array}{c}
     \sqrt2\, e^t  \\
                0 \\
      \end{array}
     \right)
\in \sN_i,\qquad
g_-=\left(
\begin{array}{c}
               0 \\
     \sqrt2\, e^{-t} \\
       \end{array}
     \right)\in \sN_{-i}.
\end{equation}
Consider also,
\begin{equation}\label{e-89-ext}
    \begin{aligned}
A x&=i\frac{dx}{dt},\\
\dom(A)&=\left\{x(t)=\left(
                       \begin{array}{c}
                         x_1(t) \\
                         x_2(t) \\
                       \end{array}
                     \right)
\,\Big|\,x_1(t),\,x_2(t) -\text{abs. cont.},\right.\\
&\left. x'_1(t)\in L^2_{(-\infty,0]},\, x'_2(t)\in L^2_{[0,\infty)},\,x_1(0-)=-x_2(0+)\right\}.\\
    \end{aligned}
\end{equation}
Clearly, $g_+-g_-\in\dom(A)$ and hence $A$ is  a self-adjoint extension of $\dA$ satisfying the conditions of Hypothesis \ref{setup}. Furthermore,
\begin{equation}\label{e-90-T}
    \begin{aligned}
T x&=i\frac{dx}{dt},\\
\dom(T)&=\left\{x(t)=\left(
                       \begin{array}{c}
                         x_1(t) \\
                         x_2(t) \\
                       \end{array}
                     \right)
\,\Big|\,x_1(t),\,x_2(t) -\text{abs. cont.},\right.\\
&\left. x'_1(t)\in L^2_{(-\infty,0]},\, x'_2(t)\in L^2_{[0,\infty)},\,x_2(0+)=0\right\}.\\
    \end{aligned}
\end{equation}
By construction, $T$ is a quasi-self-adjoint extension of $\dA$ parameterized by a von Neumann parameter $\kappa=0$ that satisfies the conditions of Hypothesis \ref{setup}. Using direct check we obtain
\begin{equation}\label{e-91-T-star}
    \begin{aligned}
T^* x&=i\frac{dx}{dt},\\
\dom(T^*)&=\left\{x(t)=\left(
                       \begin{array}{c}
                         x_1(t) \\
                         x_2(t) \\
                       \end{array}
                     \right)
\,\Big|\,x_1(t),\,x_2(t) -\text{abs. cont.},\right.\\
&\left. x'_1(t)\in L^2_{(-\infty,0]},\, x'_2(t)\in L^2_{[0,\infty)},\,x_1(0-)=0\right\}.\\
    \end{aligned}
\end{equation}
Similarly one finds
\begin{equation}\label{e-92-adj}
    \begin{aligned}
\dot A^* x&=i\frac{dx}{dt},\\
\dom(\dot A^*)&=\left\{x(t)=\left(
                       \begin{array}{c}
                         x_1(t) \\
                         x_2(t) \\
                       \end{array}
                     \right)
\,\Big|\,x_1(t),\,x_2(t) -\text{abs. cont.},\right.\\
&\left. x'_1(t)\in L^2_{(-\infty,0]},\, x'_2(t)\in L^2_{[0,\infty)}\right\}.\\
    \end{aligned}
\end{equation}
Then $\calH_+=\dom(\dA^\ast)=W^1_2(-\infty,0]\oplus W^1_2[0,\infty)$, where $W^1_2$ is a Sobolev space.  Construct a rigged Hilbert space
$$\begin{aligned}
&\calH_+ \subset \calH \subset\calH_-\\
&=W^1_2(-\infty,0]\oplus W^1_2[0,\infty)\subset L^2_{(-\infty,0]}\oplus L^2_{[0,\infty)}\subset (W^1_2(-\infty,0]\oplus W^1_2[0,\infty))_-
\end{aligned}
$$
and consider operators
\begin{equation}\label{e-93-bA}
\begin{aligned}
\bA x&=i\frac{dx}{dt}+i x(0+)\left[\delta(t+)-\delta(t-)\right],\\
\bA^\ast x&=i\frac{dx}{dt}+i x(0-)\left[\delta(t+)-\delta(t-)\right],
\end{aligned}
\end{equation}
where $x(t)\in W^1_2(-\infty,0]\oplus W^1_2[0,\infty)$, $\delta(t+)$, $\delta(t-)$ are delta-functions and elements of $(W^1_2(-\infty,0]\oplus W^1_2[0,\infty))_-$ such that
$$
\delta(t+)=\left(
                       \begin{array}{c}
                         0 \\
                         \delta_2(t+) \\
                       \end{array}
                     \right), \qquad \delta(t-)=\left(
                       \begin{array}{c}
                        \delta_1(t-)  \\
                          0
                       \end{array}
                     \right),
$$
and generate functionals by the formulas
$$(x,\delta(t+))=(x_1,0)+(x_2,\delta_2(t+))=x_2(0+),$$
and
$$
(x,\delta(t-))=(x_1,\delta_1(t-))+(x_2,0)=x_1(0-).
$$
It is easy to see
that
$\bA\supset T\supset \dA$, $\bA^\ast\supset T^\ast\supset \dA,$
and
$$
\RE\bA x=i\frac{dx}{dt}+\frac{i }{2}(x(0+)+x(0-))\left[\delta(t+)-\delta(t-)\right].
$$
Clearly, $\RE\bA$ has its quasi-kernel equal to $A$ in \eqref{e-89-ext}. Moreover,
$$
\IM\bA x=\left(\cdot,\frac{1}{\sqrt 2}[\delta(t+)-\delta(t-)]\right) \frac{1}{\sqrt 2}[\delta(t+)-\delta(t-)]=(\cdot,\chi)\chi,
$$
where $\chi=\frac{1}{\sqrt 2}[\delta(t+)-\delta(t-)]$.
Now we can build
\begin{equation}\label{e6-125-mom}
\Xi= %\left(
\begin{pmatrix}%{smallmatrix}
\bA &K &1\\
&&\\
\calH_+ \subset \calH \subset\calH_- &{ } &\dC
\end{pmatrix},
%\right),
%\quad (J=-1),
\end{equation}
that is an L-system with
\begin{equation}\label{e7-62-new}
\begin{aligned}
Kc&=c\cdot \chi=c\cdot \frac{1}{\sqrt 2}[\delta(t+)-\delta(t-)], \quad (c\in \dC),\\
K^\ast x&=(x,\chi)=\left(x,  \frac{1}{\sqrt
2}[\delta(t+)-\delta(t-)]\right)=\frac{1}{\sqrt
2}[x(0+)-x(0-)],\\
\end{aligned}
\end{equation}
and $x(t)\in \calH_+= W^1_2(-\infty,0]\oplus W^1_2[0,\infty)$.

Taking into account \eqref{charsum}, \eqref{e-87-def}, and \eqref{e-87-def} we have that $s(\dA,A)(z)\equiv0$ for all $z$ in $\dC_+$. Consequently, \eqref{blog} yields $M(\dA,A)(z)\equiv i$ for all $z$ in $\dC_+$.
Therefore, $V_\Xi(z)=i$ for all $z\in\dC_+$ which implies that $\Xi$ in \eqref{e6-125-mom} is another model for a system attractor that will be called the  \textbf{momentum system attractor}.

%%%%%%%%%%%%%%%%%%%%

\section{Examples}

\subsection*{Example 1}\label{ex-4}

This example is to illustrate the coupling of two L-systems and Theorem \ref{t6-7-3}.
First, we recall the construction of an L-systems from \cite[Example 1]{BMkT}.
%%%%%%%%%%%%%
Consider the prime symmetric operator
\begin{equation}\label{e-56}
\dA x=i\frac{dx}{dt},
\end{equation}
on
$$
\dom(\dA)=\left\{x(t)\,\Big|\,x(t) -\text{abs. cont.}, x'(t)\in L^2_{[0,\ell]},\, x(0)=x(\ell)=0\right\}.\\
$$
Its   (normalized) deficiency vectors  are (see \cite{AG93}, \cite{BMkT})
\begin{equation}\label{e-57}
g_+=\frac{\sqrt2}{\sqrt{e^{2\ell}-1}}e^t\in \sN_i,\qquad g_-=\frac{\sqrt2}{\sqrt{1-e^{-2\ell}}}e^{-t}\in \sN_{-i}.
\end{equation}
Also,
\begin{equation}\label{e-58'}
A x=i\frac{dx}{dt}
\end{equation}
on
$$
\dom(A)=\left\{x(t)\,\Big|\,x(t) -\text{abs. cont.}, x'(t)\in L^2_{[0,\ell]},\, x(0)=-x(\ell)\right\}.
$$
is  a self-adjoint extension of $\dA$ satisfying the conditions of Hypothesis \ref{setup}. The Liv\v{s}ic characteristic function $s(z)$ for the pair $(\dA,A)$ is defined and equal %(see \cite{AG93})
\begin{equation}\label{e1-1}
s(z)=\frac{e^\ell-e^{-i\ell z}}{1-e^\ell e^{-i\ell z}}.
\end{equation}
Consider the operator
\begin{equation}\label{e1-3}
    \begin{aligned}
T x&=i\frac{dx}{dt}
    \end{aligned}
\end{equation}
on
$$
\dom(T)=\left\{x(t)\,\Big|\,x(t)-\text{abs. cont.}, x'(t)\in L^2_{[0,\ell]},\,  x(0)=0\right\}.
$$
By construction, $T$ is a quasi-self-adjoint extension of $\dA$ parameterized by a von Neumann parameter $\kappa=e^{-\ell}$.
The triple of operators $( \dot A, T, A)$ satisfies the conditions of Hypothesis \ref{setup} since $|\kappa|=e^{-\ell}<1$.
The characteristic function $S(z)$ for the triple $(\dA, T, A)$
\begin{equation}\label{e-64}
S(z)=e^{i\ell z}.
\end{equation}
By the direct check one gets
\begin{equation}\label{e1-3star}
T^* x=i\frac{dx}{dt}
\end{equation}
on
$$
\dom(T)=\left\{x(t)\,\Big|\,x(t)-\text{abs. cont.}, x'(t)\in L^2_{[0,\ell]},\,  x(\ell)=0\right\}.
$$
It was shown in \cite[Example 1]{BMkT} that
\begin{equation}\label{e7-83}
\dA^\ast x=i\frac{dx}{dt}
\end{equation}
on
$$
\dom(\dA^\ast)=\left\{x(t)\,\Big|\,x(t) \text{- abs. continuous}, x'(t)\in L^2_{[0,\ell]}\right\}.\\
$$
Then $\calH_+=\dom(\dA^\ast)=W^1_2$ is a Sobolev space with scalar
product
\begin{equation}\label{e-product}
(x,y)_+=\int_0^\ell x(t)\overline{y(t)}\,dt+\int_0^\ell
x'(t)\overline{y'(t)}\,dt.
\end{equation}
 Construct rigged Hilbert space
$W^1_2\subset L^2_{[0,\ell]}\subset (W_2^1)_-$
and consider operators
\begin{equation}\label{e7-84}
\begin{aligned}
\bA x&=i\frac{dx}{dt}+i x(0)\left[\delta(t)-\delta(t-\ell)\right],\\
\bA^\ast x&=i\frac{dx}{dt}+i x(l)\left[\delta(t)-\delta(t-\ell)\right],
\end{aligned}
\end{equation}
where $x(t)\in W_2^1$, $\delta(t)$, $\delta(t-\ell)$ are delta-functions and elements of %%@
$(W^1_2)_-$ that generate functionals by the formulas
$(x,\delta(t))=x(0)$ and $(x,\delta(t-\ell))=x(\ell)$. It is easy to see
that
$\bA\supset T\supset \dA$, $\bA^\ast\supset T^\ast\supset \dA,$
and
$$
\RE\bA x=i\frac{dx}{dt}+\frac{i }{2}(x(0)+x(\ell))\left[\delta(t)-\delta(t-\ell)\right].
$$
Clearly, $\RE\bA$ has its quasi-kernel equal to $A$ in \eqref{e-58'}. Moreover,
$$
\IM\bA x=\left(\cdot,\frac{1}{\sqrt 2}[\delta(t)-\delta(t-\ell)]\right) \frac{1}{\sqrt 2}[\delta(t)-\delta(t-\ell)]=(\cdot,\chi)\chi,
$$
where $\chi=\frac{1}{\sqrt 2}[\delta(t)-\delta(t-\ell)]$.
Now we can build
\begin{equation*}%\label{e6-125}
\Theta= %\left(
\begin{pmatrix}%{smallmatrix}
\bA &K &1\\
&&\\
W_2^1\subset L^2_{[0,\ell]}\subset (W^1_2)_- &{ } &\dC
\end{pmatrix},
%\right),
%\quad (J=-1),
\end{equation*}
that is an L-system with
\begin{equation}\label{e7-62}
\begin{aligned}
Kc&=c\cdot \chi=c\cdot \frac{1}{\sqrt 2}[\delta(t)-\delta(t-l)], \quad (c\in \dC),\\
K^\ast x&=(x,\chi)=\left(x,  \frac{1}{\sqrt
2}[\delta(t)-\delta(t-l)]\right)=\frac{1}{\sqrt
2}[x(0)-x(l)],\\
\end{aligned}
\end{equation}
and $x(t)\in W^1_2$.
%%%%%%%%%%%
Now consider two L-systems of the type constructed above. These are
\begin{equation*}%\label{e6-125}
\Theta_1=
\begin{pmatrix}%{smallmatrix}
\bA_1 &K_1 &1\\
&&\\
W_2^1([0,\gamma])\subset L^2_{[0,\gamma]}\subset (W^1_2)_-([0,\gamma]) &{ } &\dC
\end{pmatrix},
\end{equation*}
and
\begin{equation*}%\label{e6-125}
\Theta_2=
\begin{pmatrix}%{smallmatrix}
\bA_2 &K_2 &1\\
&&\\
W_2^1([\gamma,\ell])\subset L^2_{[\gamma,\ell]}\subset (W^1_2)_-([\gamma,\ell]) &{ } &\dC
\end{pmatrix}.
\end{equation*}
Before we give the description of the operators $\bA_k$ and $K_k$, ($k=1,2$), we note that $\gamma\in(0,\ell)$ and $[0,\ell]=[0,\gamma]\cup[\gamma,\ell]$. The main
operators of $T_1$ and $T_2$ of the L-systems $\Theta_1$ and $\Theta_2$ are defined as follows
\begin{equation}\label{e-T1}
T_1 x=i\frac{dx}{dt},
\end{equation}
on
$$
\dom(T_1)=\left\{x(t)\,\Big|\,x(t)-\text{abs. cont.}, x'(t)\in L^2_{[0,\gamma]},\,  x(0)=0\right\},
$$
and
\begin{equation}\label{e-T2}
T_2 x=i\frac{dx}{dt},
\end{equation}
on
$$
\dom(T_2)=\left\{x(t)\,\Big|\,x(t)-\text{abs. cont.}, x'(t)\in L^2_{[\gamma,\ell]},\,  x(\gamma)=0\right\}.
$$
Their maximal symmetric parts are respectively
\begin{equation}\label{e-dA1dA2}
\dA_1 x=i\frac{dx}{dt}
\end{equation}
on
$$
\dom(\dA_1)=\left\{x(t)\,\Big|\,x(t) -\text{abs. cont.}, x'(t)\in L^2_{[0,\gamma]},\, x(0)=x(\gamma)=0\right\},\\
$$
and
$$
\dA_2 x=i\frac{dx}{dt}
$$
on
$$\dom(\dA_2)=\left\{x(t)\,\Big|\,x(t) -\text{abs. cont.}, x'(t)\in L^2_{[\gamma,\ell]},\, x(\gamma)=x(\ell)=0\right\}.
$$
As it was shown in \cite[Example 1]{BMkT}, the von Neumann parameters of $T_1$ and $T_2$ are $\kappa_1=e^{-\gamma}$ and $\kappa_2=e^{-\ell+\gamma}$, respectively.
Moreover,
    \begin{align}
\bA_1 x&=i\frac{dx}{dt}+i x(0)\left[\delta(t)-\delta(t-\gamma)\right],\label{e-bA1bA2}
\\
\bA^\ast_1 x&=i\frac{dx}{dt}+i x(\gamma)\left[\delta(t)-\delta(t-\gamma)\right], \nonumber\\
\bA_2 x&=i\frac{dx}{dt}+i x(\gamma)\left[\delta(t-\gamma)-\delta(t-\ell)\right],\nonumber\\
\bA^\ast_2 x&=i\frac{dx}{dt}+i x(\ell)\left[\delta(t-\gamma)-\delta(t-\ell)\right],\nonumber
    \end{align}

The quasi-kernels of $\RE\bA_1$ and $\RE\bA_2$ are given by
\begin{equation}\label{e-hat}
\hat A_1 =i\frac{dx}{dt}
\end{equation}
on
$$
\dom(\hat A_1)=\left\{x(t)\,\Big|\,x(t) -\text{abs. cont.}, x'(t)\in L^2_{[0,\gamma]},\, x(0)=-x(\gamma)\right\}.\\
$$
and
$$
\hat A_2 x=i\frac{dx}{dt},\;
$$
on
$$
\dom(\hat A_2)=\left\{x(t)\,\Big|\,x(t) -\text{abs. cont.}, x'(t)\in L^2_{[\gamma,\ell]},\, x(\gamma)=-x(\ell)\right\},
$$
respectively. Operators $K_1$ and $K_2$ are found via the formulas
\begin{equation}\label{e-K1K2}
\begin{aligned}
K_1c&=c\cdot \chi_1=c\cdot \frac{1}{\sqrt 2}[\delta(t)-\delta(t-\gamma)], \quad (c\in \dC),\\
K_2c&=c\cdot \chi_2=c\cdot \frac{1}{\sqrt 2}[\delta(t-\gamma)-\delta(t-\ell)], \quad (c\in \dC).
\end{aligned}
\end{equation}

It was shown in \cite[Example 4.4]{MT10} that the coupling of operators $T_1$ and $T_2$ is defined as follows
$$T=T_1\uplus T_2,$$
where
\begin{equation}\label{e-T}
  T x=i\frac{dx}{dt},
  \end{equation}
on
$$
\dom(T)=\left\{x(t)\,\Big|\,x(t)-\text{abs. cont.}, x'(t)\in L^2_{[0,\ell]},\,  x(0)=0\right\}.
$$
As one can see from \cite[Example 1]{BMkT}, the von Neumann parameter $\kappa$ of the operator $T$ in \eqref{e-T} is
$$
\kappa=e^{-\ell}=e^{-\gamma}\cdot e^{-\ell+\gamma}=\kappa_1\cdot\kappa_2.
$$
 We will use this $T$ to construct the coupling of our L-systems $\Theta_1$ and $\Theta_2$. Following \cite[Example 1]{BMkT}, we notice that the operator $\dA$ defined by \eqref{e-56} is the maximal symmetric part of $T$ and $T^*$ and the self-adjoint extension $A$ of $\dA$ defined by \eqref{e-58'} can serve as the quasi-kernel of $\RE\bA$, where $\bA$ is given by \eqref{e7-84}. Thus we have
\begin{equation}\label{e6-125}
\Theta= \Theta_1\cdot\Theta_2=
\begin{pmatrix}
\bA &K &1\\
&&\\
W_2^1\subset L^2_{[0,\ell]}\subset (W^1_2)_- &{ } &\dC
\end{pmatrix},
\end{equation}
where operator $K$ is given via \eqref{e7-62} and $T$ and $\bA$ are described above. Following the procedure of \cite[Example 1]{BMkT} or, alternatively, using calculations of $S(\dA_1,T_1,A_1)$ and $S(\dA_2,T_2,A_2)$ and formula \eqref{e-60} one finds that
\begin{equation}\label{e-84-new}
W_{\Theta_1}(z)=e^{-i\gamma z}\quad\textrm{ and }\quad W_{\Theta_2}(z)=e^{-i(\ell-\gamma) z}.
\end{equation}

Trivially then
$$
W_{\Theta}(z)=e^{-i\ell z}=e^{-i\gamma z}\cdot e^{-i(\ell-\gamma) z}=W_{\Theta_1}(z)\cdot W_{\Theta_2}(z),
$$
which confirms the conclusion of Theorem \ref{t6-7-3}. On a side note, \eqref{e6-3-6} implies that the corresponding impedance functions of L-systems $\Theta_1$, $\Theta_2$, and $\Theta$ are
$$
V_{\Theta_1}(z)=i\frac{e^{-i\gamma z}-1}{e^{-i\gamma z}+1}, \quad V_{\Theta_2}(z)=i\frac{e^{-i(\ell-\gamma) z}-1}{e^{-i(\ell-\gamma) z}+1},\quad V_{\Theta}(z)=i\frac{e^{-i\ell z}-1}{e^{-i\ell z}+1},
$$
respectively. Consequently,
$$
V_{\Theta_1}(i)=i\frac{1-e^{-\gamma}}{1+e^{-\gamma }}, \quad V_{\Theta_2}(i)=i\frac{1-e^{-\ell+\gamma }}{1+e^{-\ell+\gamma }},\quad V_{\Theta}(i)=i\frac{1-e^{-\ell}}{1+e^{-\ell }},
$$
confirming that $V_{\Theta_1}(z)\in\sM_{\kappa_1}$, $V_{\Theta_2}(z)\in\sM_{\kappa_2}$, and $V_{\Theta}(z)\in\sM_{\kappa}$ where $\kappa_1=e^{-\gamma}$, $\kappa_2=e^{-\ell+\gamma}$,
and $\kappa=\kappa_1\cdot\kappa_2=e^{-\ell}$. This illustrates Theorem \ref{t28}.

%%%%%%%%%%%%%%%
\subsection*{Example 2}\label{ex-5}

Here we will rely on Example 1  to illustrate the coupling of two L-systems that obey Hypothesis \ref{setup-1}.
Let $\dA$ be the symmetric operator defined by \eqref{e-56} but choose
\begin{equation}\label{e-58''}
    A x=i\frac{dx}{dt}
\end{equation}
on
$$
\dom(A)=\left\{x(t)\,\Big|\,x(t) -\text{abs. cont.}, x'(t)\in L^2_{[0,\ell]},\, x(0)=x(\ell)\right\}.
$$

that is  another self-adjoint extension of $\dA$ (with periodic conditions) satisfying the conditions of Hypothesis \ref{setup-1}.
Take the operator $T$ to be  still defined via \eqref{e1-3} and having a von Neumann parameter $\kappa=e^{-\ell}$.
The triple of operators $(\dA, T, A)$ then satisfies the conditions of Hypothesis \ref{setup-1}.
According to \eqref{e-49-1} and \eqref{e-64}, the characteristic function $S(z)$ for the triple $(\dA, T, A)$ is given by
\begin{equation}\label{e-64'}
S(z)=-e^{i\ell z}.
\end{equation}
Repeating the steps of Example 1 we construct the rigged Hilbert space
$W^1_2\subset L^2_{[0,\ell]}\subset (W_2^1)_-$
and consider operators
\begin{equation}\label{e7-84'}
\begin{aligned}
\bA x&=i\frac{dx}{dt}+i x(0)\left[\delta(t)+\delta(t-\ell)\right],\\
\bA^\ast x&=i\frac{dx}{dt}-i x(l)\left[\delta(t)+\delta(t-\ell)\right],
\end{aligned}
\end{equation}
where $x(t)\in W_2^1$, $\delta(t)$, $\delta(t-\ell)$ are delta-functions and elements of %%@
$(W^1_2)_-$ that generate functionals by the formulas
$(x,\delta(t))=x(0)$ and $(x,\delta(t-\ell))=x(\ell)$. It is easy to see
that
$\bA\supset T\supset \dA$, $\bA^\ast\supset T^\ast\supset \dA,$
and
$$
\RE\bA x=i\frac{dx}{dt}+\frac{i }{2}(x(0)-x(\ell))\left[\delta(t)+\delta(t-\ell)\right].
$$
Clearly, $\RE\bA$ has its quasi-kernel equal to $A$ in \eqref{e-58''}. Moreover,
$$
\IM\bA x=\left(\cdot,\frac{1}{\sqrt 2}[\delta(t)+\delta(t-\ell)]\right) \frac{1}{\sqrt 2}[\delta(t)+\delta(t-\ell)]=(\cdot,\chi)\chi,
$$
where $\chi=\frac{1}{\sqrt 2}[\delta(t)+\delta(t-\ell)]$.
Now we can build
\begin{equation*}%\label{e6-125}
\Theta= %\left(
\begin{pmatrix}%{smallmatrix}
\bA &K &1\\
&&\\
W_2^1\subset L^2_{[0,\ell]}\subset (W^1_2)_- &{ } &\dC
\end{pmatrix},
%\right),
%\quad (J=-1),
\end{equation*}
that is an  L-system with
\begin{equation}\label{e7-62'}
\begin{aligned}
Kc&=c\cdot \chi=c\cdot \frac{1}{\sqrt 2}[\delta(t)+\delta(t-l)], \quad (c\in \dC),\\
K^\ast x&=(x,\chi)=\left(x,  \frac{1}{\sqrt
2}[\delta(t)+\delta(t-l)]\right)=\frac{1}{\sqrt
2}[x(0)+x(l)],\\
\end{aligned}
\end{equation}
and $x(t)\in W^1_2$.
%%%%%%%%%%%

Now we consider two such L-systems of the type \eqref{e7-62'} that satisfy Hypothesis \ref{setup-1}. These are
\begin{equation*}%\label{e6-125}
\Theta_1=
\begin{pmatrix}%{smallmatrix}
\bA_1 &K_1 &1\\
&&\\
W_2^1([0,\gamma])\subset L^2_{[0,\gamma]}\subset (W^1_2)_-([0,\gamma]) &{ } &\dC
\end{pmatrix},
\end{equation*}
and
\begin{equation*}%\label{e6-125}
\Theta_2=
\begin{pmatrix}%{smallmatrix}
\bA_2 &K_2 &1\\
&&\\
W_2^1([\gamma,\ell])\subset L^2_{[\gamma,\ell]}\subset (W^1_2)_-([\gamma,\ell]) &{ } &\dC
\end{pmatrix}.
\end{equation*}
As in Example 1 the main operators of $T_1$ and $T_2$ of the L-systems $\Theta_1$ and $\Theta_2$ are defined by \eqref{e-T1}-\eqref{e-T2}
and their maximal symmetric parts are given respectively by \eqref{e-dA1dA2}.
Moreover,

    \begin{align}
\bA_1 x&=i\frac{dx}{dt}+i x(0)\left[\delta(t)+\delta(t-\gamma)\right],\label{e-bA1bA2'}\\
\bA^\ast_1 x&=i\frac{dx}{dt}-i x(\gamma)\left[\delta(t)+\delta(t-\gamma)\right],\nonumber\\
\bA_2 x&=i\frac{dx}{dt}+i x(\gamma)\left[\delta(t-\gamma)+\delta(t-\ell)\right],\nonumber\\
\bA^\ast_2 x&=i\frac{dx}{dt}-i x(\ell)\left[\delta(t-\gamma)+\delta(t-\ell)\right].\nonumber
    \end{align}

The quasi-kernels of $\RE\bA_1$ and $\RE\bA_2$ are given by
\begin{equation}\label{e-hat'}
   \hat A_1 x=i\frac{dx}{dt}
   \end{equation}
   on
   $$
\dom(\hat A_1)=\left\{x(t)\,\Big|\,x(t) -\text{abs. cont.}, x'(t)\in L^2_{[0,\gamma]},\, x(0)=x(\gamma)\right\},
$$
and
$$
\hat A_2 =i\frac{dx}{dt}
$$
on
$$
\dom(\hat A_2)=\left\{x(t)\,\Big|\,x(t) -\text{abs. cont.}, x'(t)\in L^2_{[\gamma,\ell]},\, x(\gamma)=x(\ell)\right\},
$$
respectively. Operators $K_1$ and $K_2$ are found via the formulas
\begin{equation}\label{e-K1K2'}
\begin{aligned}
K_1c&=c\cdot \chi_1=c\cdot \frac{1}{\sqrt 2}[\delta(t)+\delta(t-\gamma)], \quad (c\in \dC),\\
K_2c&=c\cdot \chi_2=c\cdot \frac{1}{\sqrt 2}[\delta(t-\gamma)+\delta(t-\ell)], \quad (c\in \dC).
\end{aligned}
\end{equation}
Following the procedure of \cite[Example 1]{BMkT} or, alternatively, using calculations of $S(\dA_1,T_1,A_1)$ and $S(\dA_2,T_2,A_2)$ and formulas \eqref{e-49-1}, \eqref{e-60} one finds that
\begin{equation}\label{e-84-new'}
W_{\Theta_1}(z)=-e^{-i\gamma z} \quad\textrm{ and }\quad  W_{\Theta_2}(z)=-e^{-i(\ell-\gamma) z}.
\end{equation}

As it was shown in Example 1  the coupling $T=T_1\uplus T_2$ of operators $T_1$ and $T_2$ is defined by \eqref{e-T}. We are going to  use this $T$ to construct the coupling of our L-systems $\Theta_1$ and $\Theta_2$. Formation of the coupling of two L-systems that do not obey the conditions of Hypothesis \ref{setup}
is explained in Remark 7. Let us, however, note first that
$$
W_{\Theta_1}(z)\cdot W_{\Theta_2}(z)=(-e^{-i\gamma z})\cdot (-e^{-i(\ell-\gamma) z})=e^{-i\ell z}=W_{\Theta}(z),
$$
where $\Theta$ is the L-system given by \eqref{e6-125} in Example 1. Then applying the coupling construction of Remark \ref{r-7}  we conclude that the L-system $\Theta$ of the form \eqref{e6-125} satisfying the conditions of Hypothesis \ref{setup} is indeed the coupling of two L-systems $\Theta_1$ and $\Theta_2$ that obey  Hypothesis \ref{setup-1}. Among other things, this Example confirms that the conditions of Hypothesis \ref{setup-1} are not invariant under the operation of coupling.

Similarly to Example 1 we note that \eqref{e6-3-6} implies that the corresponding impedance functions of L-systems $\Theta_1$, $\Theta_2$, and $\Theta$ are
$$
V_{\Theta_1}(z)=i\frac{-e^{-i\gamma z}-1}{-e^{-i\gamma z}+1}, \quad V_{\Theta_2}(z)=i\frac{-e^{-i(\ell-\gamma) z}-1}{-e^{-i(\ell-\gamma) z}+1},\quad V_{\Theta}(z)=i\frac{e^{-i\ell z}-1}{e^{-i\ell z}+1},
$$
respectively. Consequently,
$$
V_{\Theta_1}(i)=i\frac{1+e^{-\gamma}}{1-e^{-\gamma }}, \quad V_{\Theta_2}(i)=i\frac{1+e^{-\ell+\gamma }}{1-e^{-\ell+\gamma }},\quad V_{\Theta}(i)=i\frac{1-e^{-\ell}}{1+e^{-\ell }},
$$
confirming that $V_{\Theta_1}(z)\in\sM_{\kappa_1}^{-1}$, $V_{\Theta_2}(z)\in\sM_{\kappa_2}^{-1}$, and $V_{\Theta}(z)\in\sM_{\kappa}$ where $\kappa_1=e^{-\gamma}$, $\kappa_2=e^{-\ell+\gamma}$, and $\kappa=\kappa_1\cdot\kappa_2=e^{-\ell}$. This illustrates Corollary \ref{c-19}.

%%%%%%%%%%%%%%%

\subsection*{Example 3}\label{ex-3}

This example is to illustrate the coupling of two L-systems and Corollary \ref{c-29}. Consider two L-systems similar to the ones constructed in  Example 1. These are
\begin{equation*}%\label{e6-125}
\Theta_1=
\begin{pmatrix}%{smallmatrix}
\bA_1 &K_1 &1\\
&&\\
W_2^1([0,\gamma])\subset L^2_{[0,\gamma]}\subset (W^1_2)_-([0,\gamma]) &{ } &\dC
\end{pmatrix},
\end{equation*}
and
\begin{equation*}%\label{e6-125}
\Theta_2=
\begin{pmatrix}%{smallmatrix}
\bA_2 &K_2 &1\\
&&\\
W_2^1([\gamma,\ell])\subset L^2_{[\gamma,\ell]}\subset (W^1_2)_-([\gamma,\ell]) &{ } &\dC
\end{pmatrix}.
\end{equation*}
Once again, we note that $\gamma\in(0,\ell)$ and $[0,\ell]=[0,\gamma]\cup[\gamma,\ell]$. The main
operators of $T_1$ and $T_2$ of the L-systems $\Theta_1$ and $\Theta_2$ are defined as follows
\begin{equation}\label{e-T1'}
    T_1 x=i\frac{dx}{dt},
 \end{equation}
 on
 $$
\dom(T_1)=\left\{x(t)\,\Big|\,x(t)-\text{abs. cont.}, x'(t)\in L^2_{[0,\gamma]},\,  x(\gamma)=e^{\gamma}x(0)\right\},
$$
and
\begin{equation}\label{e-T2'}
T_2 x=i\frac{dx}{dt},
\end{equation}
on
$$
\dom(T_2)=\left\{x(t)\,\Big|\,x(t)-\text{abs. cont.}, x'(t)\in L^2_{[\gamma,\ell]},\,  x(\gamma)=0\right\}.
 $$
Their maximal symmetric parts are given by \eqref{e-dA1dA2}.

As we have shown in \cite[Example 1 and 2]{BMkT}, the von Neumann parameters of $T_1$ and $T_2$ are $\kappa_1=0$ and $\kappa_2=e^{-\ell+\gamma}$, respectively.
Moreover,
\begin{equation}\label{e-bA1bA2''}
    \begin{aligned}
\bA_1 x&=i\frac{dx}{dt}+i \frac{x(\gamma)-e^\gamma x(0)}{e^\gamma-1} \left[\delta(t-\gamma)-\delta(t)\right],\\
\bA^\ast_1 x&=i\frac{dx}{dt}+i \frac{x(0)-e^{\gamma} x(\gamma)}{e^\gamma-1} \left[\delta(t-\gamma)-\delta(t)\right],\\
\bA_2 x&=i\frac{dx}{dt}+i x(\gamma)\left[\delta(t-\gamma)-\delta(t-\ell)\right],\\
\bA^\ast_2 x&=i\frac{dx}{dt}+i x(\ell)\left[\delta(t-\gamma)-\delta(t-\ell)\right].
    \end{aligned}
\end{equation}
The quasi-kernels of $\RE\bA_1$ and $\RE\bA_2$ are given by \eqref{e-hat}.
 Operators $K_1$ and $K_2$ are found via the formulas
\begin{equation}\label{e-K1K2''}
\begin{aligned}
K_1c&=c\cdot \chi_1=c\cdot \sqrt{\frac{e^\ell+1}{2(e^\ell-1)}}\,[\delta(t)-\delta(t-\gamma)], \quad (c\in \dC),\\
K_2c&=c\cdot \chi_2=c\cdot \frac{1}{\sqrt 2}[\delta(t-\gamma)-\delta(t-\ell)], \quad (c\in \dC).
\end{aligned}
\end{equation}
First, we are going to follow the procedure developed in \cite{MT10} to construct  the coupling of operators $T_1$ and $T_2$, that is $T=T_1\uplus T_2$. This procedure begins with construction of the symmetric operator $\dA=T|_{\dom(T)\cap\dom( T^*)}$. Following the steps we obtain
\begin{equation}\label{e-108-A}
    \dA f=\left(
            \begin{array}{cc}
              i\frac{d}{dt} & 0 \\
              0 & i\frac{d}{dt} \\
            \end{array}
          \right)\left[
                   \begin{array}{c}
                     f_1 \\
                     f_2 \\
                   \end{array}
                 \right],\quad f=\left[
                   \begin{array}{c}
                     f_1 \\
                     f_2 \\
                   \end{array}
                 \right]\in\dom(\dA),
\end{equation}
where
\begin{equation}\label{e-109-dA}
\begin{aligned}
    \dom(\dA)=&\left\{ f=\left[\begin{array}{c}
                     f_1 \\
                     f_2 \\
                   \end{array}\right],\; f_1(\gamma)=e^\gamma f_1(0),\;f_2(\ell)=0,\right.\\
                  &\left. f_2(\gamma)=-e^{-\gamma}\sqrt{e^{2\gamma}-1} f_1(\gamma)   \right\}.
    \end{aligned}
\end{equation}
Direct calculations give the adjoint to $\dA$ operator
\begin{equation}\label{e-109-Astar}
    \dA^* f=\left(
            \begin{array}{cc}
              i\frac{d}{dt} & 0 \\
              0 & i\frac{d}{dt} \\
            \end{array}
          \right)\left[
                   \begin{array}{c}
                     f_1 \\
                     f_2 \\
                   \end{array}
                 \right],\quad f=\left[
                   \begin{array}{c}
                     f_1 \\
                     f_2 \\
                   \end{array}
                 \right]\in\dom(\dA^*),
\end{equation}
where
\begin{equation}\label{e-109-dA-star}
    \dom(\dA^*)=\left\{ f=\left[
                   \begin{array}{c}
                     f_1 \\
                     f_2 \\
                   \end{array}
                 \right],\; e^\gamma f_1(\gamma)-f_1(0)+ \sqrt{e^{2\gamma}-1} f_2(\gamma)=0 \right\}.
\end{equation}
In order to find the deficiency vectors $G_+$ and $G_-$ of $\dA$ we use \eqref{GGG}where (see \cite{MT10})
$$
\tan\alpha=\frac{1}{\kappa_2}
\sqrt{\frac{1-\kappa_2^2}{1-\kappa_1^2}}=\sqrt{e^{2(\ell-\gamma)}-1}.
$$
The above value gives
$$
\cos\alpha=e^{\gamma-\ell},\quad \sin\alpha=\frac{\sqrt{e^{2(\ell-\gamma)}-1}}{e^{\ell-\gamma}},
$$
and consequently
\begin{equation}\label{e-G+G-}
    G_+=\left[
          \begin{array}{c}
            \frac{\sqrt2 e^{\gamma-\ell}}{\sqrt{e^{2\gamma}-1}} \\
            \\
            -\sqrt2 e^{-\ell} \\
          \end{array}
        \right]e^t,\quad
        G_-=\left[
          \begin{array}{c}
            \frac{\sqrt2 e^{\gamma}}{\sqrt{e^{2\gamma}-1}} \\
            \\
            0 \\
          \end{array}
        \right]e^{-t}.
\end{equation}
Using the description for $\dom(\dA^*)$ we also find a general deficiency vector
\begin{equation}\label{e-Gz}
    G_z=\left[
          \begin{array}{c}
            \sqrt{e^{2\gamma}-1} \\
            \\
            e^{iz\gamma}-e^\gamma \\
          \end{array}
        \right]e^{-izt}.
\end{equation}
Having $G_+$ and $G_-$ written explicitly in \eqref{e-G+G-} we can construct the reference self-adjoint extension $\hat A$ of $\dA$ that will become the quasi-kernel of the real part
of the state-space operator  $\bA$ of the coupling system $\Theta$ we are building. We use the condition $G_+-G_-\in\dom(\hat A)$ and derive the boundary conditions taking into account that $\dom(\hat A)\subset\dom(\dA^*)$. This yields
\begin{equation}\label{e-hatA}
        \hat A f=\left(
            \begin{array}{cc}
              i\frac{d}{dt} & 0 \\
              0 & i\frac{d}{dt} \\
            \end{array}
          \right)\left[
                   \begin{array}{c}
                     f_1 \\
                     f_2 \\
                   \end{array}
                 \right],\quad f=\left[
                   \begin{array}{c}
                     f_1 \\
                     f_2 \\
                   \end{array}
                 \right]\in\dom(\hat A),
\end{equation}
where
\begin{equation}\label{e-dom-hatA}
    \begin{aligned}
    \dom(\hat A)&=\left\{ f=\left[
                   \begin{array}{c}
                     f_1 \\
                     f_2 \\
                   \end{array}
                 \right]\in\dom(\dA^*),\; \sqrt{e^{2\gamma}-1} f_1(\gamma)+e^\gamma f_2(\gamma)=f_2(\ell) ,\right.\\
                 &\left. \qquad \, \sqrt{e^{2\gamma}-1}f_2(\ell)+f_1(\gamma)=e^{\gamma}f_1(0) \right\}.
    \end{aligned}
\end{equation}
Now we use \eqref{e-68-new} to define $T=T_1\uplus T_2$. Here
\begin{equation}\label{e-TT}
        T f=\left(
            \begin{array}{cc}
              i\frac{d}{dt} & 0 \\
              0 & i\frac{d}{dt} \\
            \end{array}
          \right)\left[
                   \begin{array}{c}
                     f_1 \\
                     f_2 \\
                   \end{array}
                 \right],\quad f=\left[
                   \begin{array}{c}
                     f_1 \\
                     f_2 \\
                   \end{array}
                 \right]\in\dom(T),
\end{equation}
where
\begin{equation*}\label{e-domTT}
    \dom(T)=\left\{ f=\left[
                   \begin{array}{c}
                     f_1 \\
                     f_2 \\
                   \end{array}
                 \right],\;  f_1(\gamma)=e^\gamma f_1(0), f_2(\gamma)=-e^{-\gamma} \sqrt{e^{2\gamma}-1}\, f_1(\gamma) \right\}.
\end{equation*}
By direct calculations we also get
\begin{equation}\label{e-TTstar}
        T^* f=\left(
            \begin{array}{cc}
              i\frac{d}{dt} & 0 \\
              0 & i\frac{d}{dt} \\
            \end{array}
          \right)\left[
                   \begin{array}{c}
                     f_1 \\
                     f_2 \\
                   \end{array}
                 \right],\quad f=\left[
                   \begin{array}{c}
                     f_1 \\
                     f_2 \\
                   \end{array}
                 \right]\in\dom(T),
\end{equation}
where
\begin{equation}\label{e-domTTstar}
     \begin{aligned}
    \dom(T^*)&=\left\{f=\left[
                   \begin{array}{c}
                     f_1 \\
                     f_2 \\
                   \end{array}
                 \right],\;  e^\gamma f_1(\gamma)- f_1(0)+\sqrt{e^{2\gamma}-1}f_2(\gamma)=0,\right.
                 \\&\left. f_2(\ell)=0 \right\}.
    \end{aligned}
\end{equation}
Clearly, $\dA\subset T\subset \dA^*$ and $\dA\subset T^*\subset \dA^*$.

We know that according to \eqref{mult-for} the von Neumann parameter $\kappa$ of $T$ is
$$
\kappa=\kappa_1\cdot\kappa_2=0\cdot e^{-\ell+\gamma}=0.
$$
 We will use this $T$ to construct the coupling $\Theta$ of our L-systems $\Theta_1$ and $\Theta_2$. First, we need to write out the formulas for the $(*)$-extension $\bA$ of $T$ and its adjoint  $\bA^*$ such that $\RE\bA$ has the quasi-kernel $\hat A$. In order to do that we use the general formulas describing $(*)$-extensions of a differential operator (see \cite[Theorem 10.1.1]{ABT}).  After evaluating the appropriate coefficients we obtain
\begin{equation*}\label{e-bA}
    %\begin{aligned}
        \bA f=\left(
            \begin{array}{cc}
              i\frac{d}{dt} & 0 \\
              0 & i\frac{d}{dt} \\
            \end{array}
          \right)\left[
                   \begin{array}{c}
                     f_1 \\
                     f_2 \\
                   \end{array}
                 \right]
                %\\&
                -i(e^\gamma f_1(0)+(e^\gamma\sqrt{e^{2\gamma}-1}-1)f_1(\gamma)+e^{2\gamma}f_2(\gamma))\chi
          %\end{aligned}
\end{equation*}
and its adjoint
\begin{equation*}\label{e-bA-star}
    \begin{aligned}
        \bA^* f&=\left(
            \begin{array}{cc}
              i\frac{d}{dt} & 0 \\
              0 & i\frac{d}{dt} \\
            \end{array}
          \right)\left[
                   \begin{array}{c}
                     f_1 \\
                     f_2 \\
                   \end{array}
                 \right]-i(\sqrt{e^{2\gamma}-1} f_1(0)-e^\gamma\sqrt{e^{2\gamma}-1}f_1(\gamma)\\
&-(e^{2\gamma}-1)f_2(\gamma)-(\sqrt{e^{2\gamma}-1}+e^\gamma)f_2(\ell))\chi,
\end{aligned}
\end{equation*}
where
\begin{equation}\label{e-chi}
    \chi=\frac{1}{\sqrt2}\left[
                   \begin{array}{c}
                     (\sqrt{e^{2\gamma}-1}-e^\gamma)\delta(t)+\delta(t-\gamma) \\
                     \delta(t-\gamma)+(\sqrt{e^{2\gamma}-1}+e^\gamma)\delta(t-\ell) \\
                   \end{array}
                 \right],
\end{equation}
and $f=\left[
                   \begin{array}{c}
                     f_1 \\
                     f_2 \\
                   \end{array}\right]\in\calH_+=\dom(\dA^*)$.
At this point we have all the ingredients to compose the coupling
\begin{equation*}%\label{e6-125}
\Theta= \Theta_1\cdot\Theta_2=
\begin{pmatrix}
\bA &K &1\\
&&\\
\calH_+\subset L^2_{[0,\gamma]} \oplus L^2_{[\gamma,\ell]}\subset \calH_- &{ } &\dC
\end{pmatrix},
\end{equation*}
where operator $Kc=c\cdot\chi$ and $\chi$ is given by \eqref{e-chi}.

Now we have enough data to find $s(z)=s(\dA,\hat A)(z)$, where $\hat A$ is the self-adjoint extension $A$ of $\dA$ uniquely defined
by \eqref{refop}. Hence, \eqref{charsum}  after some tedious calculations yields for $z\in \bbC_+$
\begin{equation}\label{e-sAA}
s(z)=\frac{z-i}{z+i}\cdot \frac{(G_z, G_-)}{(G_z, G_+)}=e^{iz\ell}\cdot\frac{e^{-iz\ell}-e^\gamma}{e^\gamma-e^{iz\ell}}.
\end{equation}
We know that since $\kappa=0$ then (see \cite{MT-S}, \cite{BMkT}) $s(z)=-S(z)(=-S(\dA,T,\hat A)(z))$. Let $s_1(z)=s_1(\dA_1,\hat A_1)(z)$, $S_1(z)=S_1(\dA_1,T_1,\hat A_1)(z)$, and
$S_2(z)=$ \break $S_2(\dA_2,T_2,\hat A_2)(z)$.
 Then, using \eqref{e1-1}, \eqref{e-84-new}, and \eqref{e-sAA} give
$$
\begin{aligned}
s(z)&=-S(z)=-S_1(z)\cdot S_2(z)=s_1(z)\cdot S_2(z)=s_1(z)\cdot \frac{1}{W_{\Theta_2}(z)}\\
&=\frac{e^\gamma-e^{-i\gamma z}}{1-e^\gamma e^{-i\gamma z}}\cdot e^{-i(\ell-\gamma) z}=\frac{1}{W_{\Theta_1}(z)}\cdot \frac{1}{W_{\Theta_2}(z)}.
\end{aligned}
$$
This confirms that $W_\Theta(z)=W_{\Theta_1}(z)\cdot W_{\Theta_2}(z)$ and illustrates Theorem \ref{t6-7-3}. Moreover, this example also confirms the ``absorbtion" property of the coupling mentioned in Corollaries \ref{c-29} and \ref{c-30}. Clearly, the fact that $\kappa_1=0$ in the L-system $\Theta_1$ caused $\kappa=0$ in the coupling L-system $\Theta$ and as a result the impedance function of the coupling $V_\Theta(z)$ belongs to the Donoghue class $\sM$.

\appendix
\section{Functional model of a prime  triple}\label{A1}

In this Appendix we are going to explain the construction of a  functional model for a prime  dissipative triple.\footnote{We call a triple $(\dot A, \whA , A)$
 a prime triple if $\dot A$ is a prime symmetric operator.} parameterized by the characteristic function and obtained in \cite{MT-S}.

Given a contractive, analytic in $\bbC_+$ function $s$ that satisfies the Liv\v{s}ic criterion \cite[Theorem 15]{L}, introduce the function
$$
M(z)=\frac1i\cdot\frac{s(z)+1}{s(z)-1},\quad z\in \bbC_+,
$$
so that
$$
M(z)=\int_\bbR \left
(\frac{1}{\lambda-z}-\frac{\lambda}{1+\lambda^2}\right )
d\mu(\lambda), \quad z\in \bbC_+,
$$
for some infinite Borel measure with
$$
\int_\bbR\frac{d\mu(\lambda)}{1+\lambda^2}=1.
$$

In the Hilbert space $L^2(\bbR;d\mu)$ introduce
  the multiplication (self-adjoint)
operator by the  independent variable $\cB$
 on
\begin{equation}\label{nacha1}
\dom(\cB)=\left \{f\in \,L^2(\bbR;d\mu) \,\bigg | \, \int_\bbR
\lambda^2 | f(\lambda)|^2d \mu(\lambda)<\infty \right \},
\end{equation} denote by  $\dot \cB$  its
 restriction
on
\begin{equation}\label{nacha2}
\dom(\dot \cB)=\left \{f\in \dom(\cB)\, \bigg | \, \int_\bbR
f(\lambda)d \mu(\lambda) =0\right \},
\end{equation}
and let
 $\whB $ be   the dissipative restriction of the operator  $(\dot \cB)^*$
 on
\begin{equation}\label{nacha3}
\dom(\whB )=\dom (\dot \cB)\dot +\linspan\left
\{\,\frac{1}{\cdot -i}- S(i)\frac{1}{\cdot +i}\right \}.
\end{equation}

We will refer to the triple  $(\dot \cB,   \whB ,\cB)$ as {\it  the model  triple } in the Hilbert space $L^2(\bbR;d\mu)$.
It was established in \cite{MT-S} that a triple $(\dot A,\whA ,A)$ with the characteristic function $S$ is unitarily equivalent to the model triple
$(\dot \cB,   \whB ,\cB)$ in the Hilbert space $L^2(\bbR;d\mu)$ whenever the underlying symmetric operator $\dot A$ is prime.
The triple $(\dot \cB,   \whB ,\cB)$ will therefore be called {\it the functional model} for $(\dot A,A,\whA )$.

For the resolvents of the model dissipative operator $\whB$ and the self-adjoint (reference) operator $\cB$ from the model  triple $(\dot \cB, \whB , \cB)$ in the Hilbert space
$L^2(\bbR;d\mu) $ one gets  the following resolvent formula (see \cite{MT-S}).
\begin{equation}\label{e-46-res-form-1}
(\whB - zI )^{-1}=(\cB- zI )^{-1}-p(z)(\cdot\, ,g_{\overline{z}})g_z,
\end{equation}
with
\begin{equation}\label{e-47-res-form-2}
p(z)=\left (M(\dot \cB, \cB)(z)+i\frac{\kappa+1}{\kappa-1}\right )^{-1},\quad z\in\rho(\whB )\cap \rho(\cB).
\end{equation}
Here $M(\dot \cB,\cB)$ is the Weyl-Titchmarsh function associated with the pair  $(\dot \cB, \cB)$ continued to the lower half-plane by the Schwarz reflection
principle, and the deficiency elements $g_z$ are given by
\begin{equation}\label{e-52-def}
g_z(\lambda)=\frac{1}{\lambda-z}, \quad
\,\, \text{$\mu$-a.e. }.
\end{equation}
The same resolvent formula takes place \cite{MT-S} for a given  triple $(\dot A, \whA , A)$ satisfying  Hypothesis \ref{setup}.

%%%%%%%%%%%%%%%

\end{document}